\documentclass[a4paper,11pt]{amsart}

\usepackage{a4wide}
\usepackage{times, epsfig}
\usepackage{bbm}
\usepackage{amsfonts, amscd, amsmath, amssymb, amsthm}
\usepackage{mathtools}
\usepackage[utf8]{inputenc}

\usepackage{graphicx}
\usepackage{enumitem}
\usepackage{color}

\newcommand{\scalar}[2]{\langle#1\,,#2\rangle}

\newcommand{\norm}[1]{\|#1\|}
\newcommand{\normm}[1]{|\negthinspace\|#1\|\negthinspace|}

\newcommand{\A}{\mathcal{A}}

\renewcommand{\H}{\mathcal{H}}
\renewcommand{\L}{\mathcal{L}}

\newcommand{\D}{\mathcal{D}}
\newcommand{\Df}{\mathcal{D}_{\mathrm{Fin}}}

\newcommand{\bQ}{\widehat{Q}}
\newcommand{\bD}{\widehat{\D}}
\renewcommand{\d}{\mathrm{d}}
\renewcommand{\a}{\alpha} 

\newcommand{\1}{\mathbb{I}}
\newcommand{\CC}{\mathbb{C}}
\newcommand{\R}{\mathbb{R}}
\newcommand{\Z}{\mathbb{Z}}
\newcommand{\N}{\mathbb{N}}

\newcommand{\map}[5][\phantom]{
\begin{array}{ccl}
\mathllap{#1{\,:\,}}{#2}&\to&{#3}\\
          {#4}&\mapsto&{#5}
\end{array}
}

\newtheorem{theorem}{Theorem}
\newtheorem{corollary}[theorem]{Corollary}
\newtheorem{proposition}[theorem]{Proposition}
\newtheorem{lemma}[theorem]{Lemma}
\newtheorem{example}[theorem]{Example}

\theoremstyle{definition}
\newtheorem{definition}[theorem]{Definition}

\theoremstyle{remark}
\newtheorem*{remark}{Remark}

\newlist{Henumerate}{enumerate}{1}
\setlist[Henumerate, 1]{label = \textbf{H\arabic*:}, ref = {{\bf H\arabic*}}}

\numberwithin{equation}{section}
\numberwithin{theorem}{section}

\title[Representation of quadratic forms and orthogonal additivity]{
Representation of non-semibounded quadratic forms and orthogonal additivity
}
\author[A.~Ibort]{Alberto Ibort}
\author[J.G.~Llavona]{Jos\'e G.~Llavona}
\author[F.~Lled\'o]{Fernando Lled\'{o}}
\author[J.M.~P\'erez-Pardo]{Juan Manuel P\'erez-Pardo}
\address[A.~Ibort, F.~Lled\'o, J.M.~P\'erez-Pardo]{
Universidad Carlos~III de Madrid, Avda. de la Universidad 30, 28911 Legan\'es (Madrid), Spain.
}
\address[A.~Ibort, F.~Lled\'o, J.M.~P\'erez-Pardo]{
Instituto de Ciencias Matem\'{a}ticas (CSIC - UAM - UC3M - UCM), Nicol\'as Cabrera, 13-15, 28049 Cantoblanco (Madrid), Spain.
}
\address[J.G.~Llavona]{
Universidad Complutense de Madrid, Plaza Ciencias, 3, 28040 Madrid, Spain.
}

\date{\today}

\thanks{
A.I.~was partially supported by Spanish Ministry of Economy and
Competitiveness through project DGI MTM2017-84098-P, by the
\emph{Severo Ochoa} Program for Centers of Excellence in R\&D
(SEV-2015-0554) and by QUITEMAD+, S2013/ICE-2801.
J.G.L.~was partially supported by Spanish Ministry of Economy and
Competitiveness through project \hbox{DGI MTM2014-54692-P}.
F.L.~was partially supported by Spanish Ministry of Economy and
Competitiveness through project DGI MTM2017-84098-P and by the
\emph{Severo Ochoa} Program for Centers of Excellence in R\&D
(SEV-2015-0554).
J.M.P. is supported by QUITEMAD+, S2013/ICE-2801 and the ``Juan de la Cierva - Incorporaci\'on" Project 2018/00002/001 and is partially supported by Spanish Ministry of Economy and
Competitiveness through project DGI MTM2017-84098-P}



\begin{document}

\maketitle

\begin{abstract}
A representation theorem for non-semibounded Hermitian quadratic forms in terms of a (non-semibounded) self-adjoint operator is proven.
The main assumptions are closability of the Hermitian quadratic form, the direct integral structure of the underlying Hilbert space and orthogonal additivity.
We apply this result to several examples, including the position operator in quantum mechanics and quadratic forms invariant under a unitary
representation of a separable locally compact group. The case of invariance under a compact group is also discussed in detail.
\end{abstract}

\section{Introduction}

Integral representation theorems for functionals defined on Banach or Hilbert spaces are fundamental in solving many problems and applications related to them.
In addition, the representation theorem of semibounded, closed and Hermitian quadratic forms is the cornerstone of many representation theorems that can be given for quadratic forms (see, e.g., \cite[Chapter~VI~\S2]{kato95}).
It goes back to the pioneering work in the 1950s by Friedrichs, Kato, Lax, Milgram, and others (see comments to Section~VIII.6 in \cite{reed-simon-1}).
In its simplest version, this result provides a representation of the quadratic form in terms of a densely defined, semibounded and self-adjoint operator.
Hence, the spectral theorem applied to the self-adjoint operator gives finally the desired integral representation as the following example shows
(see, e.g., \cite[\S~VIII.6, Example 2]{reed-simon-1}).

\begin{example}\label{ex:trepresentsq}
Let $T$ be a self-adjoint operator on a Hilbert space $\H$
with dense domain \hbox{$\D(T)\subset\H$}. Let $E\colon \mathrm{Borel}(\R)\to \mathrm{Proj}(\H)$ be the resolution of the identity associated to the operator $T$,
where $\mathrm{Borel}(\R)$ and $\mathrm{Proj}(\H)$ denote the $\sigma$-algebra of Borel sets on $\R$ and the lattice of orthogonal projections on $\H$, respectively.
For any $\Phi,\Psi\in\H$ define a complex measure on Borel sets of $\R$ by
$$\nu_{\Phi,\Psi}(\sigma) : = \scalar{\Phi}{E(\sigma)\Psi},\quad \sigma \in\mathrm{Borel}(\R)\;.$$
We will denote the corresponding positive measure simply as $\nu_{\Phi}(\sigma) := \nu_{\Phi,\Phi}(\sigma)$;
note that from the properties of the resolution of the identity we have $\nu_{\Phi}(\R)=\|\Phi\|^2$.
The domain of the operator $T$ can be characterised as
$$\D(T) = \left\{ \Phi \in \H \mid \int_\R  \lambda^2 \d\nu_\Phi(\lambda) < \infty \right\}$$
and we consider the following quadratic form defined on $\D(T)$
\begin{equation}\label{eq:defQsa}
    Q(\Phi,\Psi):=\scalar{\Phi}{T\Psi}=\int_\R \lambda\, \d\nu_{\Phi,\Psi}(\lambda)\;.
\end{equation}
\end{example}

The key hypothesis of the representation theorem mentioned before are semiboundedness and closability of the Hermitian quadratic form.
While the latter is shown to be a necessary condition (see Definition~\ref{def:strongrepresentability} and Theorem~\ref{thm:reprimpliesclosable})
the previous example shows that semiboundedness is not. In fact, the quadratic form $Q$ given in Eq.~\eqref{eq:defQsa}
is represented by a, in general,  non-semibounded self-adjoint operator $T$, see also the example of the quantum mechanical position operator on $L^2(\R)$ in Subsection~\ref{subsec:position}).

Over the years, many authors have tried to remove the condition of semiboundedness and have proved that it can be weakened in various ways.
For instance, in terms of sectorial quadratic forms (see \cite{kato55, kato95}), accretive quadratic forms (cf.\ \cite{mcintosh1})
or indefinite metric spaces, which appear naturally in perturbation theory of, e.g., Dirac-Coulomb operators  (cf.\ \cite{grubisic13, schmitz15}).
None of these generalisations can be applied to the general situation described by the example above without further restrictions on the operator $T$.

In Theorem~\ref{thm:mainthm1} we give sufficient conditions for a closable, non-semibounded quadratic form to be strongly representable by a self-adjoint operator $T$ (see Definition~\ref{def:strongrepresentability} for details) and allowing an integral representation as in Eq.~\eqref{eq:defQsa}. These analytical conditions include a direct integral decomposition of the underlying Hilbert space (cf.\ conditions \textbf{H1}, \textbf{H2} and \textbf{H3} in Section~\ref{sect:representation}). In particular, these hypothesis allow to write the following formula for the extension $(\bQ,\bD)$ of the given quadratic form:
\[
 \bQ(\Phi) = \scalar{|T|^{1/2}\Phi}{(\operatorname{sgn}T) |T|^{1/2}\Phi}=\int_\R \lambda \;\d\nu_\Phi(\lambda) \;,\quad\Phi \in \bD=\D(|T|^{1/2})\;,
\]
where $\operatorname{sgn}T : = E((0,\infty)) - E((-\infty,0))$ and 
$\nu_\Phi(\cdot) = \norm{E(\cdot)\Phi}^2$ is the natural positive spectral measure. 
Conversely, if the quadratic form is strongly representable by a self-adjoint operator, then its spectral decomposition $T=\int_\R \lambda\, \d E(\lambda)$ allows to give a decomposition of the Hilbert space verifying the conditions \textbf{H1}, \textbf{H2} and \textbf{H3}.

The representation theorem of quadratic forms was instrumental in developing the perturbation theory of self-adjoint operators and
had enormous importance in Mathematical Physics, in particular in the development of Quantum Mechanics (see also Section~10 of \cite{Simon18} for a review of the central importance of quadratic forms in the analysis of self-adjointness of unbounded operators in non-relativistic Quantum Mechanics). Moreover,
due to the min-max principle, quadratic forms also play a fundamental role in the study of spectral properties
(like, e.g.,  existence of spectral gaps) for operators like Laplacians on manifolds as well as on
quantum and discrete graphs (see \cite{balmaseda19,balmaseda19b,lledo-post08a,lledo-post08b,fabila-lledo-post18,fabila-lledo19,post:16}).
It is remarkable that non-semibounded quadratic forms play an important role in the characterisation of significant operators
like the Laplace-Beltrami operator on Riemannian manifolds with boundary (and dimension $\geq 2$) due to the non-trivial
boundary term (see for instance  the recent results in \cite{Ib14}) or Dirac-like operators in any dimension \cite{Ib15, Perez17}.

Most classical representation theorems for quadratic forms require some analytical conditions to define a norm on the domain
which guarantees the existence of an unbounded self-adjoint operator representing it (see, e.g., \cite[\S~VIII.6]{reed-simon-1}
or \cite[\S~4.4]{davies:95}). The existence of a norm is typically guaranteed by the semiboundedness of
the quadratic form (or more generally assuming the quadratic form to be sectorial).
But if the quadratic form is not sectorial, in particular not semibounded, then one has to look for alternative analytic conditions in relation
with the domain of the quadratic form. There have been various approaches in the literature to treat non-semibounded quadratic forms starting with the classical
ones by McIntosh \cite{mcintosh2, mcintosh3}.
More recently, the articles \cite{dbella:17, corso:19} study representation theorems for sesquilinear forms which are $q$-closed and solvable.
This notion of $q$-closedness gives the domain of the sesquilinear form a Banach space structure and considers a Banach-Gelfand triple associated with it.
Also Arendt and Ter Elst study in \cite{arendt12} representation theorems for non-closable sectorial quadratic forms.
Moreover, non-semibounded quadratic forms also
appear naturally when describing generalised Friedrichs extensions of non-semibounded symmetric operators (see, e.g., \cite{mcintosh2,dsnoo:14}).
In this article we propose to remove the closedness condition as a primary requirement and generalise the notion of closable quadratic form
to the non-semibounded case. In particular, we will not require a priori the existence of a norm on the domain of the quadratic form.

In a very different setting, a remarkable activity has taken place around the problem of finding linear representations, and eventually integral representations, for continuous homogeneous polynomials on Banach lattices. The first result in this direction was obtained by Sundaresan \cite{Sun91} who obtained a representation theorem for polynomials on $\ell_{p}$ and on $L^{p}$.
Given a Banach lattice of functions on a measure space $(\A,\Sigma(\A))$ with measure $\mu$,  a natural model for a continuous homogeneous polynomial of degree $n$ is
\begin{equation}\label{integral_form}
P(f) = \int_\A f^n (x)\, \xi (x)\,  \d\mu (x)\, ,
\end{equation}
where $\xi$ is an integrable function.    Clearly such polynomials satisfy the orthogonal additivity property:
\begin{equation}\label{additive_orthogonality}
P(f + g) = P(f) + P(g) \, ,
\end{equation}
whenever the functions $f$ and $g$ have disjoint supports.
Benyamini, Llavona and Lasalle showed in \cite{benyamini06} that continuous orthogonally additive polynomials on Banach
lattices can be represented in the form \eqref{integral_form}. P{\'e}rez-Garc{\'\i}a and
Villanueva \cite{Perez Garcia-Villanueva 2005} solved independently the case of spaces of continuous
functions on compact spaces.  Carando, Lassalle and Zalduendo \cite{Carando-Lassalle-Zalduendo 2006}
found an independent proof for the case $C(K)$ and Ibort, Linares and Llavona \cite{Ibort-Linares-Llavona 2009}
gave another one for the case $\ell_{p}$. Such representation theorem can be extended to the more general situation
of Riesz spaces \cite{Jonge-Van Rooij 1970}. Following ideas started by Buskes and van Rooij
(see \cite{Buskes-Van Rooij 2004}), a similar representation theorem was obtained recently by Ibort, Linares and Llavona \cite{Ib12}.

In this article, we will concentrate on non-continuous non-semibounded quadratic polynomials because of its interest,
both conceptually and in applications, and we will extend the classical representation theorem by using an orthogonality property inspired in the additive orthogonality of homogeneous polynomials on Banach lattices discussed before. To implement orthogonal additivity
in our context we will assume that the underlying Hilbert space has the structure of a direct integral
(see Section~\ref{sec:QFandDI}). This assumption is quite natural if one aims to represent the quadratic form in terms of
(non-semibounded) self-adjoint operators because the spectral theorem guarantees
the existence of such direct integral decomposition (see, e.g., Section~7.3 in \cite{Hall13}).
The decomposition of the Hilbert space in terms of a direct integral and the notion of orthogonal additivity provides a natural and consistent
way in which to split a non-semibounded quadratic form into semibounded ones. This makes it possible to extend results  on semibounded forms
to non-semibounded ones.
In particular, this has applications in numerical analysis, where quadratic forms play a prominent role \cite{Ib13,Lop17}.

Another source of inspiration for
this article, in particular, to justify the direct integral decomposition of the Hilbert space is the theory of unitary
representation of groups (see, e.g., \cite{mackey:76}). Recall that using classical results in the theory of von
Neumann algebras one may decompose any unitary representation of
a separable locally compact group $G$ on a Hilbert space as a direct integral of primary representations.
In the context of compact groups this direct integral becomes simply a direct sum. Quadratic forms invariant under
a given unitary representation of a group $G$ were analysed in \cite{Ib14b}.
We will address in Subsection~\ref{subsec:groups} the question of representability of non-semibounded and $G$-invariant
quadratic forms.

The article is organised as follows: In the next section we introduce the necessary definitions and results in relation to
quadratic forms and direct integrals. We also generalise the notions of closability and representability to non-semibounded
quadratic forms. These concepts will be central for this article.
In Section~\ref{sect:orthogonal} we introduce, assuming that the underlying Hilbert space has a direct integral structure over
a measurable space $(\A,\Sigma(\A))$ with measure $\mu$, the notions of orthogonally additive and countably orthogonally additive quadratic
form. Under the assumption of closability of the quadratic form these notions turn out to be equivalent. We also interpret the
quadratic form as a measure on $\Sigma(\A)$ and its Radon-Nikodym derivative (with respect to $\mu$) as a quadratic form density.
In Section~\ref{sect:representation} we prove the main representation results for a class of non-semibounded Hermitian quadratic forms
satisfying certain conditions which we resume under \textbf{H1}, \textbf{H2} and \textbf{H3} (cf.\ Theorems~\ref{thm:mainthm1} and \ref{thm:mainthm2}).
For this family of quadratic forms the classical representation theorems do not apply in general.
This class contains quadratic forms densely defined on the direct integral Hilbert space,
which are closable, orthogonally additive and, for simplicity, we also assume that
the measure space underlying the direct integral is point supported. We base our result on
a consistent splitting procedure of the quadratic form into strongly representable quadratic form densities and the spectral theorem.
Finally, in Section~\ref{sec:examples} we apply our results to some of the examples mentioned above. It includes
the case of the position operator in Quantum Mechanics. This example is prototypical since the spectral
theorem guarantees that any self-adjoint operator is equivalent to a multiplication operator.
We conclude with the analysis of the representability of non-semibounded quadratic forms invariant under a unitary representation of a
group in Subsection~\ref{subsec:groups}. In the last section we briefly comment on some uniqueness issues around the representation theorem.

{\bf Acknowledgements:} We would like to thank the referees of this article for a careful reading of the manuscript and many useful remarks and suggestions. In particular, for suggesting a substantial and elegant simplification of the proof of Theorem~\ref{thm:mainthm1}.

\section{Quadratic forms and direct integrals}\label{sec:QFandDI}
We begin introducing basic definitions and results on quadratic forms and direct integrals that will be needed later.
Some standard references on quadratic forms and their relation to unbounded operators are,
e.g.,~\cite[Chapter~VI]{kato95}, \cite[Section~VIII.6]{reed-simon-1}, \cite[Section~4.4]{davies:95} or \cite[Chapter~10]{schmuedgen12}.
For the notion of direct integral of Hilbert spaces we refer to \cite[Part II, Chapter 1]{dixmier81} (see also Chapter~14 in \cite{Kadison-Ringrose-II}
or Section~7.3 in \cite{Hall13} and references cited therein).

\begin{definition}\label{def:qf}
Let $\D$ be a dense subspace in a complex separable Hilbert space $\H$. Denote by $Q\colon\mathcal{D}\times\mathcal{D}\to \mathbb{C}$
a sesquilinear form which is anti-linear in the first entry and linear in the second.
The sesquilinear form $Q$ is \textbf{Hermitian} if
\[
  Q(\Phi,\Psi)=\overline{Q(\Psi,\Phi)}\;,\quad \Phi,\Psi\in\mathcal{D}\;.
\]

Given a Hermitian sesquilinear form $Q$, the \textbf{quadratic form} associated with it (and also denoted by $Q$) on the
the same domain $\mathcal D$ is its evaluation on the diagonal, i.e., $Q\colon\D\to\R$ with $Q(\Phi):=Q(\Phi,\Phi)$ for
any $\Phi\in\mathcal{D}$. The quadratic form is \textbf{semibounded from below} if there is a constant $m \in \R$ such that
\[
  Q(\Phi)\geq m \norm{\Phi}^2\;,\; \Phi \in \mathcal{D}\;.
\]
We call $m$ the {\bf semibound} of $Q$.
If  $Q(\Phi)\geq 0$, $\Phi\in\mathcal{D}$, we say that $Q$ is \textbf{positive}. The quadratic form is
\textbf{semibounded from above} if $-Q$ is semibounded from below. We say that the quadratic form is \textbf{non-semibounded} if it
is not semibounded from below and from above.
\end{definition}

Note that a quadratic form $Q$ satisfies for any scalar $\lambda\in\CC$
\[
 Q(\lambda \Phi)=|\lambda|^2 Q(\Phi)\;,\quad \Phi\in\D\;,
\]
which implies $Q(0)=0$. Moreover, the values of the quadratic form already determine the values of the sesquilinear form
(cf.\ \cite[p.~49]{kato95}):
\begin{equation}\label{eq:quadratic-polarization}
 Q(\Phi,\Psi)=\frac{1}{4}\Big( Q(\Phi+\Psi) - Q(\Phi-\Psi) + iQ(\Phi+i\Psi) - iQ(\Phi-i\Psi)\Big)\;,\quad \Phi,\Psi\in\D\;.
\end{equation}

The definition we give next was stated by Kato in \cite[\S~1.4 of Chapter~VI]{kato95} as a characterisation of closability
in the context of sectorial quadratic forms. We will define closability as Kato but without assuming sectoriality and will show that it plays a central role when addressing orthogonal additivity and the representation theorem for non-semibounded quadratic forms.

\begin{definition}\label{def:closable}
A densely defined, Hermitian quadratic form $Q$ on a dense domain $\D\subset\H$ is \textbf{closable} if for any $\{\Phi_n\}_{n\in\mathbb{N}}\subset\D$ such that $\lim_{n\to\infty}\Phi_n=0$ and $\lim_{n, m\to\infty}Q(\Phi_n-\Phi_m)=0$ one has that $\lim_{n\to\infty}Q(\Phi_n)=0$\,.
\end{definition}

For semibounded quadratic forms one can define the notion of closed quadratic form and give a characterisation
in terms of the norm defined below, see \cite[Chapter~VI \S~1.3 Theorem~1.11]{kato95}.

\begin{definition}\label{def:graph-norm}
Let $Q$ be a Hermitian quadratic form with dense domain $\D$. If $Q$ is semibounded from below with semibound $m\in\R$ we define the \textbf{graph norm} of the quadratic form $Q$ by
$$\normm{\Phi}^2_Q=(1-m)\norm{\Phi}^2+Q(\Phi)\;, \Phi\in\D.$$
If $Q$ is semibounded from above with semibound $m\in\R$ we define the \textbf{graph norm} of the quadratic form $Q$ by
$$\normm{\Phi}^2_Q=(1+m)\norm{\Phi}^2-Q(\Phi)\;, \Phi\in\D.$$
\end{definition}

\begin{definition}\label{def:extension}
    Let $Q$ be a quadratic form densely defined on $\D\subset\H.$ We will say that $\widehat{Q}$ is an \textbf{extension} of $Q$ on the dense domain
    $\widehat{\D} \supset \D$ if $\widehat{\D} \subset \H$ and $\widehat{Q}|_{\D} = Q$.
\end{definition}

It can be proven that closable sectorial forms, and therefore closable semibounded quadratic forms, admit a smallest closed extension \cite[Chapter~VI, \S~1.4, Theorem~1.17]{kato95}. This extension is defined in a canonical way extending by continuity with respect to the graph norm. Without the semiboundedness assumption there is not a clear connection between a closable quadratic form and a closed extension of it, hence the notion of closed quadratic form does not play an important role in our context. The following theorems summarise these results including classical representation theorems on semibounded quadratic forms \cite[Chapter~VI, \S~2.1, Theorems~2.1 and 2.6]{kato95}.

\begin{theorem}\label{thm:representation}
    Let $Q$ be a semibounded, closable, Hermitian quadratic form densely defined on \hbox{$\D\subset\H$}. Then there is a minimal extension $\overline{Q}$ of $Q$ with domain $\overline{\D} := \overline{\D}^{{\normm{\cdot}_Q}}$.
\end{theorem}

\begin{definition}\label{def:core}
Let $(Q,\D)$ be a closable, semibounded, Hermitian quadratic form densely defined on $\D$ and denote by $(\overline{Q},\overline{\D})$ the minimal extension of the preceding theorem.
We will say that $\D_0\subset\overline{\D}$ is a \textbf{core} for $\overline{Q}$ (or just a core for $Q$) if $\overline{\D}_0^{{\normm{\cdot}_{\overline{Q}}}} = \overline{\D}$.
\end{definition}

\begin{theorem}\label{thm:representation-2}
    Let $Q$ be a semibounded, closable, Hermitian quadratic form densely defined on \hbox{$\D\subset\H$} and denote by $(\overline{Q},\overline{\D})$ the minimal extension of the preceding theorem.
    Then there is a unique self-adjoint operator $T$ with domain $\D(T)\subset \overline{\D}$ such that $\D(T)$ is a core for $Q$ and such that for all
    $\Phi \in \overline{\D}$ and $\Psi\in\D(T)$ one has that
    $$\overline{Q}(\Phi,\Psi) = \scalar{\Phi}{ T\Psi}.$$
\end{theorem}

To any self-adjoint operator one can associate a unique resolution of the identity $E(\cdot)$ (cf.\ \cite[Chapter~VI, Section 66, Theorems~1 and 2]{akhiezer93}).
For every $\Phi\in\H$ this resolution of the identity defines a real measure on the Borel $\sigma$-algebra of $\R$ by
    $$\nu_{\Phi}(\sigma) = \norm{E(\sigma)\Phi}^2,\quad \sigma \in \mathrm{Borel}(\R).$$
The domain of the self-adjoint operator $T$ can be characterised in terms of this family of measures as follows:
    \begin{equation}\label{eq:DT}
    \D(T) = \left\{ \Phi \in \H \;\Bigr|\;  \int_{\R} \lambda^2 \d\nu_{\Phi}(\lambda) < \infty\right\}\;.
    \end{equation}

Notice also that as a consequence of the spectral theorem one has that a form like in Theorem~\ref{thm:representation-2} can be represented using the measure $\nu_{\Phi}(\cdot)$ for $\Phi\in\D(T)$ as
    \begin{equation}\label{eq:intformulaQ}
        \overline{Q}(\Phi) = \int_\R\lambda \d\nu_\Phi(\lambda).
    \end{equation}

This result as well as in \cite[Proposition~2.2]{mcintosh3} suggests a new notion of representability. Before introducing it we will define a convenient norm.

\begin{definition}\label{def:spectralnorm}
    Let $E(\cdot)$ be the unique resolution of the identity associated to a self-adjoint operator $T$. We define the \textbf{spectral norm} given by
    $$\normm{\Phi}_E^2 : = \int_\R(1+|\lambda|)\d\nu_\Phi(\lambda),\quad\Phi\in\D(|T|^{1/2}).$$
\end{definition}

\begin{definition}\label{def:strongrepresentability}
    Let $(Q,\D)$ be a Hermitian quadratic form and let $T$ be a self-adjoint operator with domain $\D(T)$ and $E(\cdot)$ the corresponding resolution of the identity. We will say that $Q$ is \textbf{strongly representable} by $T$ if
    \begin{enumerate}
        \item There exists an extension $(\bQ,\bD)$ of $(Q,\D)$ such that $\D(T)\subset\bD$ and $\D \subset \D(|T|^{1/2})$.
        \item $\bD = \overline{\D}^{\normm{\cdot}_E}$\rule{0pt}{1em}.
        \item $\bQ(\Phi) = \int_\R\lambda \d\nu_\Phi(\lambda)$ for all $\Phi \in \bD$\rule{0pt}{1em}.
    \end{enumerate}
\end{definition}

The following proposition is a direct consequence of Definition~\ref{def:strongrepresentability} and the Spectral Theorem.

\begin{proposition}\label{prop:D(T)denseDhat}
Let $Q$ be a strongly representable quadratic form with domain $\D$ and representing self-adjoint operator $T$ with resolution of the identity $E(\cdot)$.
Then $\bD\subset\D(|T|^{1/2})$ and $\bQ(\Phi) = \scalar{|T|^{1/2}\Phi}{(\operatorname{sgn}T)|T|^{1/2}\Phi}$ for $\Phi\in\bD$,  where
$\operatorname{sgn}T := E\left((0,\infty)\right)-E\left((-\infty,0)\right)$.
\end{proposition}

We will show next that closability is a necessary condition for any quadratic form defined via a symmetric operator.

\begin{theorem}\label{thm:reprimpliesclosable}
Let $T$ be a self-adjoint operator on the dense domain $\D(T)\subset\H$ and define the Hermitian quadratic form
$Q(\Phi):=\langle \Phi,T\Phi\rangle$, $\Phi\in\D:=\D(T)$. Then $Q$ is closable.
\end{theorem}
\begin{proof}
We need
to show that given $\{\Phi_n\} \subset \D(T)$ such that $\lim_{n\to\infty} \Phi_n = 0$ and $\lim_{n, m\to\infty}Q(\Phi_n - \Phi_m) = 0$ then one has that $\lim_{n\to\infty}Q(\Phi_n) = 0$. Notice that we cannot use the closedness of the self-adjoint operator $T$ to prove the result since
$\lim_{n\to\infty}T\Phi_n$ might not exist. Indeed, it would exist (and be zero) if $\lim_{n, m \to \infty}\norm{T(\Phi_n - \Phi_m)} = 0$,
but our conditions are weaker.

Suppose that $Q(\Phi_n)$ does not converge to zero. Then there exist $\epsilon_0>0$ such that for all $N\in\mathbb{N}$ there exist $n_0 = n_0(N,\epsilon_0)>N$
with $|Q(\Phi_{n_0})|>\epsilon_0$. Using the additivity of the sesquilinear form on both entries we have
\[
 Q(\Phi_n) = Q(\Phi_n-\Phi_m)-Q(\Phi_m) + Q(\Phi_n,\Phi_m)+Q(\Phi_m,\Phi_n)\;,
\]
hence,

\begin{equation}
    Q(\Phi_n) + Q(\Phi_m) = Q(\Phi_n-\Phi_m) + 2\;\mathrm{Re}\scalar{\Phi_m}{T\Phi_n}.
\end{equation}
By assumption we have that for all $\epsilon_1 >0$ there exists $N_1$ such that for all $n, m>N_1$ we have that $|Q(\Phi_n-\Phi_m)|<\epsilon_1$. Choose $\epsilon_1 = \epsilon_0/2$ and $N = N_1$ such that $|Q(\Phi_{n_0(N_1,\epsilon_0)})|>\epsilon_0.$ We have now that
\begin{equation}
    |Q(\Phi_m) + Q(\phi_{n_0})| \leq |Q(\Phi_m-\Phi_{n_0})| + 2|\scalar{\Phi_m}{T\Phi_{n_0}}|
\end{equation}
Since $\Phi_n \to 0$ we can choose $m_0 > N_1$ such that $$2|\scalar{\Phi_{m_0}}{T\Phi_{n_0}}| \leq 2\norm{\Phi_{m_0}}\norm{T(\Phi_{n_0})} \leq \epsilon_0/2.$$
Hence
\begin{equation}\label{eq:proofclosability}
        |Q(\Phi_{m_0}) + Q(\Phi_{n_0})| < \epsilon_0.
\end{equation}
The sequence $\{ Q(\Phi_n) \}$ may alternate signs.
Suppose that $Q(\Phi_{n_0}) > 0$. Then $m_0$ can be chosen such that $Q(\Phi_{m_0})>0$. If this were not the case, this would mean that there exists some $N_2$ such that the series does not alternate sign for $n>N_2$. Then, it suffices to choose $N_1> N_2$. Now we have that Eq.~(\ref{eq:proofclosability}) implies that $|Q(\Phi_{n_0})| < \epsilon_0$, a contradiction. One can proceed analogously if $Q(\Phi_{n_0}) < 0$.
\end{proof}

The same proof of the preceding proposition also shows the following result in relation with symmetric operators.
\begin{corollary}\label{cor:symmetric}
Let $T$ be a symmetric operator on the dense domain $\D(T)\subset\H$. Define the Hermitian quadratic form
$Q(\Phi):=\langle \Phi,T\Phi\rangle$, $\Phi\in\D:=\D(T)$. Then $Q$ is closable.
\end{corollary}

Given a semibounded, Hermitian quadratic form we can consider the spectral norm of the unique self-adjoint operator associated to its smallest closed extension, cf.\ Theorem~\ref{thm:representation-2}. Such a norm is convenient because it does not need the use of the semibound explicitly in its definition. 

\begin{proposition}\label{pro:equivalentnorms}
Let $Q$ be a Hermitian quadratic form on a dense domain $\D$ which is closed and
semibounded (either from above or from below) with semibound $m\in\R$.
Let $E(\cdot)$ be the resolution of the identity associated to the unique self-adjoint operator representing the quadratic form.
Then the norms $\normm{\cdot}_E$ and $\normm{\cdot}_Q$ are equivalent.
\end{proposition}

\begin{proof}
The statement follows immediately for semibounded forms from the Closed Graph Theorem and Kato's Second Representation Theorem \cite[Theorem VI.2.23]{kato95}
\end{proof}

We prove now the following consistency result showing that the representation theorem for semibounded quadratic forms
given in Theorem~\ref{thm:representation-2}, implies strong representability.

\begin{corollary}\label{cor:Katorepisstrongrep}
Let $Q$ be a closable, semibounded, Hermitian quadratic form densely defined on $\D$. Then $Q$ is strongly representable by the unique operator $T$ determined by the closure $(\overline{Q},\overline{\mathcal{D}})$.
\end{corollary}

\begin{proof}
Take $(\overline{Q},\overline{\D})$ and $T$ as in Theorem~\ref{thm:representation-2} and let $E(\cdot)$ be the associated resolution of the identity. Define
$(\bQ,\bD):=(\overline{Q},\overline{\D})$ and, by
the equivalence of norms stated in Proposition~\ref{pro:equivalentnorms}, we conclude that
$\D\subset\overline{\D}^{\normm{\cdot}_E}=\overline{\D}^{\normm{\cdot}_Q}=\D(|T|^{1/2})$.
Finally, since $\bQ(\Phi) = \scalar{\Phi}{T\Phi}$ for all $\Phi\in\D(T)$ and since $\D(T)$ is a core for
$Q$, the integral representation in Eq.~\eqref{eq:intformulaQ} can be extended by continuity to all $\Phi\in\overline{\D}^{\normm{\cdot}_E}$.
\end{proof}


We complete this section describing some basic definitions and results on direct integrals.
Let $(\A,\Sigma(\A),\mu)$ be a measure space with $\mu$ a $\sigma$-finite, positive measure. Let $\{\H^\alpha \mid  \alpha\in\A \}$ be a family of complex
separable Hilbert spaces and we denote for any $\alpha\in\A$ by
$\scalar{\cdot}{\cdot}_\alpha$ and $\norm{\cdot}_\alpha$ the scalar product and the norm of the Hilbert space $\H^\alpha$,
respectively. An element $x$ of $\prod_{\alpha\in\A} \H^\alpha$ is a mapping $\alpha\mapsto x(\alpha)$
with the property that for each $\alpha\in\A$ one has $x(\alpha)\in\H^\alpha$. Such a mapping is called a vector field over $\A$.

\begin{definition}
The space $\mathcal{F}\subset\prod_{\alpha\in\A} \H^\alpha$  is called a \textbf{measurable field of Hilbert spaces with respect to $\mu$}
if it is a complex vector space such that:
\begin{enumerate}
  \item For all $x \in \mathcal{F}$, the function $\alpha \mapsto \norm{x(\alpha)}_\alpha$,  is $\mu$-measurable.
  \item If $z$ is a vector field over $\mathcal{A}$ such that for all $x\in\mathcal{F}$ the function $\alpha \mapsto \scalar{x(\alpha)}{z(\alpha)}_\alpha$,
    is $\mu$-measurable, then $z\in\mathcal{F}$\,.
  \item There is a countable set of vector fields
    $\{ x_1,x_2,\dots \}$ in $\mathcal{F}$ such that for all $\alpha\in\A$ the linear span of $\{x_1(\alpha),x_2(\alpha),\dots\}$ is dense in $\H^\alpha$\,. Note that this condition requires separability of the $\H^\alpha$ explicitly.
\end{enumerate}
\end{definition}

We define the following semi-norm on the space $\mathcal{F}$.
\begin{equation}\label{eq:normdirectintegral}
	\norm{x}^2:=\int_{\A}\norm{x(\alpha)}^2_\alpha\d\mu(\alpha)\;.
\end{equation}

\begin{definition}
We call the set $\{x\in\mathcal{F}\mid  \norm{x}<\infty\}$ the space of \textbf{square integrable sections} of the measurable field of Hilbert spaces.
The quotient of this space under the equivalence class $x\sim y$ defined by $\norm{x-y}=0$ is called the \textbf{direct integral}
of the measurable field of Hilbert spaces with respect to $\mu$ and is denoted by
$$\int^\oplus_\A\H^\alpha\d\mu(\alpha)\;.$$
\end{definition}

The following theorem summarises the preceding construction, cf.\ \cite[Part II, Chapter 1]{dixmier81}.

\begin{theorem}\label{thm:directintegral}
Let $(\A,\Sigma(\A),\mu)$ be a measure space with $\mu$ a $\sigma$-finite, positive measure. Let $\{\H^\alpha \mid  \alpha\in\A \}$ be a family of complex separable Hilbert spaces and let $\mathcal{F}$ be a measurable field of Hilbert spaces with respect to $\mu$. The direct integral
$$\H=\int^\oplus_\A\H^\alpha\d\mu(\alpha)$$
defines a complex Hilbert space with norm
$$\norm{x}=\left(\int_\A\norm{x(\alpha)}_\alpha^2\d\mu(\alpha)\right)^{1/2}$$
and associated scalar product
$$\scalar{x}{y}=\int_\A\scalar{x(\alpha)}{y(\alpha)}_\alpha\d\mu(\alpha)\;.$$
\end{theorem}

From these definitions some important consequences follow.

\begin{proposition}\label{prop:projections}
Let $(\A,\Sigma(\A),\mu)$ be a measure space with $\mu$ a $\sigma$-finite, positive measure.
Let $\{\H^\alpha \mid  \alpha\in\A \}$ be a $\mu$-measurable field of Hilbert spaces.
For any $\Delta\in\Sigma(\A)$ consider the corresponding measure subspace. Then
$\{\H^\alpha\mid \alpha\in\Delta \}$ is a measurable field of Hilbert spaces with respect to $\mu$
and define the direct integral over $\Delta$ by
$$\H_\Delta:=\int^\oplus_\Delta\H^\alpha\d\mu(\alpha)\;.$$
Then $\H_\Delta$ is a closed subspace of $\H$. We define $P_\Delta$ to be the orthogonal projection onto $\H_\Delta$.
Moreover, if $\mu(\Delta)=0$, then $\H_\Delta=\{0\}$.
\end{proposition}

\section{Orthogonal additivity of quadratic forms}\label{sect:orthogonal}

In this section we will define the property of orthogonal additivity. This condition is an important hypothesis of Theorem~\ref{thm:mainthm1}
(see also Theorem~\ref{thm:reverseimplication}).

\begin{definition}\label{def:oa}
Let $(\A,\Sigma(\A),\mu)$ be a measure space with positive and $\sigma$-finite measure $\mu$.
Consider a measurable field of Hilbert spaces $\{\H^\alpha| \alpha\in\A\}$ and the direct integral
$$\H = \int_\A^\oplus \H^\alpha\d\mu(\alpha)\,.$$
Let $Q$ be a Hermitian quadratic form densely defined on $\D\subset\H$. We say that $Q$ is \textbf{orthogonally additive with respect to $\mu$}
or, simply, \textbf{orthogonally additive} if the following properties hold:
\begin{enumerate}
    \item \label{item:oa1} \textit{Stability of domain:} for any $\Delta\in\Sigma(\A)$ one has that $P_\Delta\D \subset \D$, where
    $P_\Delta$ is the orthogonal projection onto the subspace $\H_\Delta$ (cf.\ Proposition~\ref{prop:projections}).
    \item \label{item:oa2} \textit{$\Sigma$-boundedness:} for any $\Phi\in\D$ there exists $M_\Phi>0$ such that for all $\Delta \in \Sigma(\A)$
        $$|Q(P_\Delta\Phi)| < M_\Phi\,,$$
        where $M_{\Phi}$ is independent of $\Delta$.
    \item \label{item:oa3} \textit{Additivity:} for any finite partition $\{\Delta_i\}_{i = 1}^N \subset \Sigma(\A)$, $N\in\mathbb{N}$, of a measurable set $\Delta\in \Sigma(\A)$ one has that for all $\Phi\in\D$
        $$Q(P_\Delta \Phi) = \sum_{i=1}^N Q(P_{\Delta_i}\Phi)\,.$$
\end{enumerate}

\end{definition}

There is a natural alternative notion related to orthogonal additivity that one can define.

\begin{definition}\label{def:countableoa}
Let $(\A,\Sigma,\mu)$ and $\H$ as in Definition~\ref{def:oa}. Let $Q$ be a Hermitian quadratic form, densely defined on $\D\subset\H$. We say that the quadratic form is \textbf{countably orthogonally additive with respect to $\mu$} or, simply, \textbf{countably orthogonally additive} if
the following properties hold:
\begin{enumerate}
    \item \textit{Stability of domain:} for any $\Delta\in\Sigma(\A)$ one has that $P_\Delta\D \subset \D$.
    \item \label{item:coa} \textit{Countable additivity:} for any countable partition $\{\Delta_i\}_{i \in \mathbb{N}} \subset \Sigma(\A)$ of a measurable set $\Delta\in \Sigma(\A)$ one has that for all $\Phi\in\D$
        \begin{equation}\label{eq:coa}
            Q(P_\Delta \Phi) = \sum_{i=1}^\infty Q(P_{\Delta_i}\Phi)\,.
        \end{equation}
\end{enumerate}
\end{definition}

Note that the convergence of the preceding series is part of the requirement. Since the union of sets is not changed under permutation of indices
any rearrangement of the series must also be convergent; in particular, the series is absolutely convergent.

The next result is a straightforward consequence of the stability of domain condition.
\begin{lemma}\label{lem:denseD}
 Let $\H = \int_\A^\oplus \H^\alpha\d\mu(\alpha)$ be a direct integral of Hilbert spaces and $\D\subset\H$ be a subspace satisfying
 $P_\Delta\D \subset \D$ for any $\Delta\in\Sigma(\A)$. Then $\D$ is dense in $\H$ if and only if $P_\Delta \D$ is dense in
 $\H_\Delta=\int_\Delta^\oplus \H^\alpha\d\mu(\alpha)$ for
 any $\Delta\in\Sigma(\A)$.
\end{lemma}
\begin{proof}
 To prove the non-obvious direction, assume $P_\Delta \D$ is not dense in $\H_\Delta$ for some $\Delta\in\Sigma(\A)$.
 Therefore, there is a nonzero $\Phi_\Delta\in \H_\Delta$ which is orthogonal to $P_\Delta \D$.
 Extending $\Phi_\Delta$ by $0$ on the complement $\Delta^c$ we obtain a nonzero vector in $\H$ which is orthogonal to $\D$, hence $\D$ is not dense in $\H$.
\end{proof}

In the rest of this section we will show that, under the assumption of closability of the quadratic form (cf.\ Definition~\ref{def:closable}),
orthogonal additivity and countable orthogonal additivity
are equivalent notions. In the following we will assume that the underlying Hilbert space $\H$ is a direct integral, i.e.,
\[
 \H = \int_\A^\oplus \H^\alpha\d\mu(\alpha)
\]
for a measure space $(\A,\Sigma(\A),\mu)$ with positive and $\sigma$-finite measure $\mu$.

\begin{lemma}\label{lem:oaimpliescoa}
    Let $Q$ be a closable and orthogonally additive quadratic form defined on $\D$ which is dense in the direct integral Hilbert space $\H$.
    Then, for any countable partition $\{\Delta_i\}_{i\in\mathbb{N}}$ of a measurable set $\Delta\in \Sigma(\A)$ and any $\Phi\in\D$ one has that
    $$\lim_{n\to\infty} Q\Big(P_{\cup_{i=n}^\infty\Delta_i}\Phi\Big) = 0\,.$$
\end{lemma}
\begin{proof}
    Let $\{\Delta_i\}_{i\in\mathbb{N}}$ be a partition of $\Delta$ and assume $\mu(\Delta_i)\not=0$, $i\in\mathbb{N}$,
    since if only finitely many $\mu(\Delta_i)$ are different from $0$ there is nothing to prove. Take $\Phi\in\D$ and
    consider the sequence
    \[
     a_n := Q\Big(P_{\cup_{i=n}^\infty\Delta_i}\Phi\Big)\;,n\in\N\;,
    \]
    which is bounded by the $\Sigma$-boundedness property~(\ref{item:oa2}) in Definition~\ref{def:oa}.
    If ${a_n}$ does not converge to zero, then there exist $\epsilon>0$ and a subsequence $\{a_{n_j}\}_{j\in\mathbb{N}}$ which can be taken to be convergent
    and such that $|a_{n_j}|>\epsilon$, $j\in\mathbb{N}$. Since $\{a_{n_j}\}_{j\in\mathbb{N}}$ is, in particular, a Cauchy sequence,
    there exists a $K>0$ such that for all $n_j, n_l$ with $n_j\leq n_l$ and $j, l >K$ we have
    \begin{align}
        \epsilon &> |a_{n_j} - a_{n_l}| = |Q(P_{\cup_{i=n_j}^\infty\Delta_i}\Phi) - Q(P_{\cup_{i=n_l}^\infty\Delta_i}\Phi)| \\
                 &= | \sum_{i=n_j}^{n_l-1}Q(P_{\Delta_i}\Phi)|= |Q(P_{\cup_{i=n_j}^\infty\Delta_i}\Phi - P_{\cup_{i=n_l}^\infty\Delta_i}\Phi)|\,,
    \end{align}
    where we have used the orthogonal additivity property in the last two equalities.
    From the properties of the direct integral of Hilbert spaces $\lim_{n\to\infty} \norm{P_{\cup_{i=n}^\infty \Delta_i} \Phi} = 0$.
    Therefore, closability of $Q$ now implies that
    $$\lim_{j\to\infty}  Q(P_{\cup_{i=n_j}^\infty\Delta_i}\Phi) = 0\,,$$
    which contradicts $|a_{n_j}|>\epsilon$.
\end{proof}

\begin{corollary}\label{cor:oaimpliescoa}
    Let $Q$ be an Hermitian quadratic form on $\D$ which is dense in the direct integral Hilbert space $\H$.
    If $Q$ is closable and orthogonally additive, then $Q$ is countably orthogonally additive.
\end{corollary}

\begin{proof}
    Let $\{\Delta_i\}_{i\in\mathbb{N}}$ be a countable partition of a measurable set $\Delta\in\Sigma(\A)$. Then, by orthogonal additivity, for any $n\in\mathbb{N}$ we have that
    $$Q(P_{\Delta}\Phi) = \sum_{i=1}^n Q(P_{\Delta_i}\Phi)+ Q(P_{\cup_{i=n+1}^\infty\Delta_i}\Phi)\,.$$
    From Lemma~\ref{lem:oaimpliescoa} we have that
    $$Q(P_{\Delta}\Phi) = \lim_{n\to\infty}\sum_{i=1}^n Q(P_{\Delta_i}\Phi)$$
    which implies countable orthogonal additivity.
\end{proof}

\begin{proposition}\label{prop:muoasesq}
Let $Q$ be a countably orthogonally additive quadratic form on $\D$ which is dense in the direct integral Hilbert space $\H$.
Let $\{\Delta_i\}_{i\in\mathbb{N}}$ be a partition of a measurable set $\Delta\in \Sigma(\A)$\,. Then, the sesquilinear form defined by the quadratic form verifies:
\begin{equation}\label{eq:mu-orthogonalsesq}
Q(P_{\Delta}\Phi,P_{\Delta}\Psi)=\sum_{i=1}^\infty Q(P_{\Delta_i}\Phi,P_{\Delta_i}\Psi)\,,\qquad \forall \Phi\,,\Psi\in\D\;.
\end{equation}
\end{proposition}

\begin{proof}
This is a direct application of countable orthogonal additivity and the polarisation identity on both sides (see Eq.~(\ref{eq:quadratic-polarization})).
\end{proof}

\begin{corollary}\label{cor:orthogonalq}
Let $Q$ be a closable and countably orthogonally additive quadratic form defined on $\D$ which is dense in the direct integral Hilbert space $\H$.
If $\Delta_1,\Delta_2\in\Sigma(\A)$ are two disjoint sets, then
$$Q(P_{\Delta_1}\Phi,P_{\Delta_2}\Psi)=0,\quad\forall \Phi,\Psi\in\D\;.$$
\end{corollary}
\begin{proof}
Given $\Delta_1,\Delta_2$ as above, define $\Delta_3:=(\Delta_1\cup\Delta_2)^c$ and
consider the partition $\{\Delta_1,\Delta_2,\Delta_3\}$ of $\A$. Then, by the preceding proposition we have
\begin{align}
Q(P_{\Delta_1}\Phi,P_{\Delta_2}\Psi)&=Q(P_{\Delta_1}^2\Phi,P_{\Delta_1}P_{\Delta_2}\Psi)+Q(P_{\Delta_2}P_{\Delta_1}\Phi,P^2_{\Delta_2}\Psi)\notag\\& \phantom{aaa}+Q(P_{\Delta_3}P_{\Delta_1}\Phi,P_{\Delta_3}P_{\Delta_2}\Psi)\notag\\
&=0\;,\notag
\end{align}
where we have used that the projections $P_{\Delta_i}\,$, $i=1,2,3\,,$ are mutually orthogonal.
\end{proof}

The property of countable orthogonal additivity of the quadratic form allows to introduce the following family of real measures
on $\Sigma(\A)$.

\begin{definition}\label{def:realmeasure}
Let $Q$ be a countably orthogonally additive Hermitian quadratic form on $\D$ which is dense on the direct integral Hilbert space $\H$.
For any $\Phi\in\D$ we define a (signed) real measure on the measure space $(\A,\Sigma(\A))$ by
$$
\map[\Omega_{\Phi}]{\Sigma(\A)}{\mathbb{R}}{\Delta}{Q(P_\Delta\Phi)}\;.
$$
Recall that this measure is countably additive by property (\ref{item:coa}) in
Definition~\ref{def:countableoa}.
\end{definition}

\begin{proposition}\label{prop:finitevariation}
Let $\Omega_\Phi$ be the real measure associated to a countably orthogonally additive Hermitian quadratic form $Q$ on $\D$ and $\Phi\in\D$.
The total variation $|\Omega_\Phi|$ of the measure $\Omega_\Phi$ is a finite measure, i.e.,
$$|\Omega_\Phi|(\A)<\infty\;.$$
\end{proposition}

\begin{proof}
This is a direct consequence of Theorem~6.4 in \cite{rudin86}. The statement follows under the assumptions that the series  on the right hand side of Equation~\eqref{eq:coa} converges and that $$|\Omega_\Phi(\Delta)| = |Q(P_{\Delta}\Phi)|< \infty$$ for any $\Phi\in\D$ and any $\Delta\in\Sigma(\A)$.
\end{proof}

\begin{corollary}\label{cor:uniformbound}
    Let $Q$ be a Hermitian quadratic form densely defined on $\D\subset\H$ and satisfying countable orthogonal additivity. Then, for all $\Phi\in\D$ there exists $M_\Phi>0$ such that for all $\Delta\in\Sigma(\A)$
    $$|Q(P_\Delta\Phi)| < M_\Phi. $$
\end{corollary}

\begin{proof}
    We have that
    $$|Q(P_\Delta\Phi)|  = |\Omega_\Phi(\Delta)| \leq  |\Omega_\Phi|(\Delta) \leq |\Omega_\Phi|(\A) <\infty\,,$$
    where we have used the properties of the total variation of a measure and the previous proposition.
\end{proof}

We can summarise the relation between the two notions of orthogonal additivity in the following theorem.

\begin{theorem}\label{thm:oa=coa}
    Let $Q$ be a closable Hermitian quadratic form on $\D$ which is dense on the direct integral Hilbert space $\H$.
    Then $Q$ is orthogonally additive if and only if it is countably orthogonally additive.
\end{theorem}

\begin{proof}
    That orthogonal additivity implies countable orthogonal additivity is shown in Corollary~\ref{cor:oaimpliescoa}. Clearly we have that (\ref{item:coa}) of Definition~\ref{def:countableoa} implies (\ref{item:oa3}) of Definition~\ref{def:oa}. In addition, Corollary~\ref{cor:uniformbound} states that countable orthogonal additivity implies (\ref{item:oa2}) of Definition~\ref{def:oa}.
\end{proof}


The fact that the Hermitian form $Q$ is defined on a direct integral Hilbert space
$ \H = \int_\A^\oplus \H^\alpha\d\mu(\alpha)$
(for a measure space $(\A,\Sigma(\A),\mu)$, with positive and $\sigma$-finite measure $\mu$)
and has a stable domain $\D$, allows the interpretation of $Q$ as a map with three arguments:
\begin{equation}\label{eq:trimap}
 Q\colon \D\times\D\times\Sigma(\A)\to \mathbb{C}\;,\quad\mathrm{where}\quad (\Phi,\Psi,\Delta)\mapsto Q(P_\Delta\Phi,P_\Delta\Psi)\;.
\end{equation}
(and similarly for the associated quadratic form $Q\colon \D\times\Sigma(\A)\to \R$). Fixing $\Phi\in\D$ we have
considered in Definition~\ref{def:realmeasure} the real measure $\Omega_\Phi$ on $(\A,\Sigma(\A))$. In the next proposition we
continue exploring properties of this measure.

\begin{proposition}\label{prop:RN}
Let $Q$ be a countably orthogonally additive quadratic form densely defined in \hbox{$\D\subset\H$} and let $\Omega_\Phi$ be the associated real measure,
for every $\Phi\in\D$. There exists a density function $\omega_\Phi\in L^1(\mu)$ such that:
$$
Q(P_\Delta\Phi) = \Omega_\Phi(\Delta)=\int_\Delta \omega_\Phi(\alpha)\d\mu(\alpha)\,, \quad\forall \Delta\in\Sigma(\A)\;.
$$
\end{proposition}

\begin{proof}
This is a direct application of the Radon-Nikodym Theorem. Indeed, $\mu$ is a $\sigma$-finite, positive measure and $\Omega_\Phi$ is a real measure with finite total variation by Proposition~\ref{prop:finitevariation}. It only remains to show that $\Omega_\Phi$ is absolutely continuous with respect to $\mu$\,. Let $\Delta\in\Sigma(\A)$ be such that $\mu(\Delta)=0$. Then, by Proposition~\ref{prop:projections}, we have $\H_\Delta=\{0\}$ and, therefore, $\Omega_\Phi(\Delta)=Q(P_{\Delta}\Phi)=0$\,.
\end{proof}

\section{Representation of non-semibounded quadratic forms}\label{sect:representation}

In this section we present a representation theorem for non-semibounded quadratic forms on a Hilbert space $\H$ which we assume to have the structure
of a direct integral $\H = \int^{\oplus}_\A\H^{\alpha}\d\mu(\alpha)$, where $(\A,\Sigma(\A),\mu)$ is a measure space with $\sigma$-finite and
positive measure $\mu$. In order to highlight the essential ideas behind the representation theorem we will make some additional assumptions
on the measure space. For a countable set $\A$, we say that the measure $\mu$
is point-supported on $(\A,\Sigma(\A))$ if $\mu$ measures all subsets of $\A$, i.e., $\Sigma(\A)=\mathcal{P}(\A)$ and for any $\Delta\subset\A$ we have
\[
 \mu(\Delta)=\sum_{x\in\Delta}\mu(\{x\})\;.
\]
It is clear that, in particular, $\mu$ is purely atomic (see, e.g., \cite{fremlin10}).

Our first two hypothesis for the representation theorem are:
\begin{Henumerate}
    \item \label{H1} $Q$ is a closable and orthogonally additive quadratic form on the domain $\D$ which is dense in the direct integral
    Hilbert space $\H$.
    \item \label{H2} The positive measure $\mu$ on $(\A,\Sigma(\A))$ is point supported and $\mu(\{\alpha\})>0$, $\alpha\in\A$.
\end{Henumerate}

\begin{definition}\label{def:qalpha}
    Let $Q$ be a quadratic form satisfying \ref{H1} and \ref{H2}, and for $\Phi\in\D$ let $\omega_\Phi\in L^1(\mu)$ be the density function of Proposition~\ref{prop:RN}. For any $\alpha\in\A$, we introduce the sesquilinear form $$\map[q_\alpha]{P_\alpha\D\times P_\alpha\D}{\mathbb{C}}{(P_\alpha\Phi,P_\alpha\Psi)}{\frac{1}{4}\sum_{k=0}^3 (-i)^k\omega_{P_\alpha(\Phi+i^k\Psi)}(\alpha)}.$$
    The restriction to the diagonal will be denoted by the same symbol $q_\alpha$ as usual, cf.\ Definition~\ref{def:qf}.
\end{definition}

\begin{remark}
Notice that it follows $$\omega_{\Phi}(\alpha) = Q(P_\alpha\Phi)/\mu(\{\alpha\})=\omega_{P_\alpha\Phi}(\alpha)=q_\alpha(P_\alpha\Phi).$$
\end{remark}

\begin{proposition}\label{pro:sector}
    Let $Q\colon\D\to\R$ be a quadratic form satisfying
    \ref{H1} and \ref{H2}. For any $\alpha\in\A$ consider the map $q_\alpha$ of Definition~\ref{def:qalpha}.
    Then $q_\alpha$ specifies a closable, Hermitian quadratic form on $P_\alpha\D$ which is dense on $\H^\alpha$.
\end{proposition}

\begin{proof}
Recall that $P_\alpha\D$ is dense in $\H^\alpha$ by Lemma~\ref{lem:denseD}.
It remains to show that $q_\alpha(\cdot)$ is closable. Let $\{\Phi_n(\alpha)\}_n\subset P_\alpha\D$ be a sequence such that $\lim_{n\to\infty}\Phi_n (\alpha)= 0$ and
$\lim_{n,m\to\infty}q_\alpha(\Phi_n(\alpha)-\Phi_m(\alpha))=0$. Since $\{\Phi_n(\alpha)\}_n$ can be extended by zero on $\A\setminus\{\alpha\}$ to a sequence $\{\widehat{\Phi}_n\}_n$
on $\D$ we also have
\[
\lim_{n,m\to\infty}Q(\widehat{\Phi}_n - \widehat{\Phi}_m )/\mu(\{\alpha\})=\lim_{n,m\to\infty}q_\alpha(\Phi_n(\alpha)-\Phi_m(\alpha))=0 \;.
\]
Now, since $\mu(\{\alpha\})>0$, by closability of $Q$ (cf.\ Definition~\ref{def:closable})
we conclude that
\[
\lim_{n\to\infty} q_\alpha(\Phi_n(\alpha))=\lim_{n\to\infty} Q(\widehat{\Phi}_n)/\mu(\{\alpha\})=0\;,
\]
hence $q_\alpha(\cdot)$ is closable.
\end{proof}

Finally we add the last hypothesis to complete the requirements of the representation theorem. We will need that the restriction of $Q$ to each $\alpha\in\A$
is strongly representable.

\begin{Henumerate}
    \setcounter{Henumeratei}{2}
    \item \label{H3}
    Let $Q\colon\D\to\R$ be a quadratic form satisfying
    \ref{H1} and \ref{H2}. Denote by $q_\alpha \colon P_\alpha\D\to \R$, $\alpha\in\A$, the family of quadratic forms of Definition~\ref{def:qalpha}.
    We assume that each of them is strongly representable by a self-adjoint operator $T_\alpha$.
\end{Henumerate}

\begin{remark} Note that if $Q$ is a Hermitian quadratic form satisfying \ref{H1}, \ref{H2} and \ref{H3}, then, for each $\alpha\in\A$
            $$
                q_\alpha\bigl(\Phi(\alpha)\bigr) = \int_{\R}\lambda \d\nu^\alpha_\Phi(\lambda)\,,\quad \Phi(\alpha) \in P_\alpha\D,
            $$
            where $\nu^\alpha_\Phi(\cdot) = \norm{E^{\alpha}(\cdot) \Phi(\alpha)}^2_\alpha$ is the positive measure on Borel sets of $\R$
            associated with the resolution of the identiy $E^\alpha(\cdot)$ of the self-adjoint operator $T_\alpha$ that strongly represents $q_\alpha$.
\end{remark}

In the next result we will integrate the family of resolutions of the identity $E^\alpha(\cdot)$ provided by a quadratic form satisfying \ref{H3} to specify a new resolution of the identity on {$\H=\int^\oplus_\A\H^\alpha\d\mu(\alpha)$} which, eventually, will be associated to
a self-adjoint operator representing the quadratic form $Q$.

\begin{proposition}\label{pro:integrateE}
Let $\{E^\alpha(\cdot)\}_{\{\alpha\in\A\}}$ be the family of resolutions of the identity associated to a quadratic form satisfying conditions \ref{H1}, \ref{H2} and \ref{H3}. For each Borel set $\sigma\subset\R$ define the projection-valued map
\[
  \sigma\mapsto E(\sigma):=\int^\oplus_\A E^\alpha(\sigma)\d\mu(\alpha)\;.
\]
The family $E(\cdot)$ is a resolution of the identity on $\H$ that defines a self-adjoint operator $T:=\int_\R\lambda dE(\lambda)$ on the domain $\D(T)= \left\{\Phi\in\H \mid \int_\R\lambda^2 \d\nu_{\Phi}(\lambda)<\infty \right\}$.
\end{proposition}
\begin{proof}
 We have to verify the defining properties of a resolution of the identity (see, e.g., \cite[Chapters~5 and 6]{Birman-Solomjak-87}). Completeness and monotonicity follow immediately from the corresponding properties of the resolutions $E^\alpha$, $\a\in\A$, and the direct integral structure. The right continuity property, i.e.,
\[
\mathrm{s}-\lim_{\delta\searrow 0}E\big((-\infty,\lambda+\delta]\big)=E\big((-\infty,\lambda]\big)\;,\quad \lambda\in\R\;,
\]
follows from the fact that we can take the strong limit inside the direct integral by the dominated convergence theorem and because
$\norm{E^\alpha\big((-\infty,\lambda]\big)\Phi(\alpha)}_\alpha \leq \norm{\Phi(\alpha)}_\alpha$ for each $\lambda\in\R$.
 \end{proof}

\begin{theorem}\label{thm:mainthm1}
Let $Q$ be a Hermitian quadratic form defined on $\D$ which is dense in the direct integral Hilbert space $\H$ and satisfying \ref{H1}, \ref{H2} and \ref{H3}; let $T$ be the self-adjoint operator associated to the resolution of the identity $E(\cdot)$ obtained in Proposition~\ref{pro:integrateE}. Assume, in addition, that $\D\subset \D (|T|^{1/2})$.
Then $Q$ is strongly representable by $T$ (see Definition~\ref{def:strongrepresentability}). The extension of $(Q,\D)$ is given by $\bD=\D(|T|^{1/2})$ and
\[
 \bQ(\Phi) = \scalar{|T|^{1/2}\Phi}{(\operatorname{sgn}T) |T|^{1/2}\Phi}=\int_\R \lambda \;\d\nu_\Phi(\lambda) \;,\quad\Phi \in \bD\;,
\]
where $\operatorname{sgn}T : = E((0,\infty)) - E((-\infty,0))$ and
$\nu_\Phi(\cdot) = \norm{E(\cdot)\Phi}^2$.
\end{theorem}

\begin{proof}
Putting $\bD:=\overline{\D}^{\normm{\cdot}_E}$, where $\normm{\cdot}_E$ is the spectral norm (cf.\ Definition~\ref{def:spectralnorm}), it is clear that $\bD\subset \D(|T|^{1/2})$. According to Definition~\ref{def:strongrepresentability} we need to show that $\D(T)\subset\bD$, $\bD= \D(|T|^{1/2})$ and that the quadratic form $(\bQ,\bD)$ is extension of $(Q,\D)$.

To show the inclusion of the domains $\D(T)\subset\bD$ recall from assumption \ref{H3} that the quadratic forms $q_\alpha$ are strongly represented by the self-adjoint operators $T_\alpha$. Therefore
for any finite $\Delta\subset\A$ one has
\[
 P_\Delta (\D(T)) \subset \overline{P_\Delta\D}^{\normm{\cdot}_E}\subset \bD\subset\D(|T|^{1/2})\;.
\]
Now if $\Phi\in\D(T)$ and since $(\A,\mu)$ is point supported there is a sequence $\Phi_n$, $n\in\N$, such that $\Phi_n\in\overline{P_{\Delta_n}\D}^{\normm{\cdot}_E}$ for some
finite $\Delta_n\subset\A$ and such that $\normm{\Phi-\Phi_n}_E\mathop{\longrightarrow}\limits^{n\to\infty} 0$. Since $\bD$ is $\normm{\cdot}_E$-closed
we conclude that $\D(T)\subset\bD$. From the latter inclusion and since $\overline{\D(T)}^{\normm{\cdot}_E} = \D(|T|^{1/2})$ we also obtain that $\bD= \D(|T|^{1/2})$.

Finally, to show that $(\bQ,\bD)$ extends $(Q,\D)$ note first that by
Proposition~\ref{prop:RN} and the strong representability of the forms $q_\alpha$, for $\Phi\in\D$ we have that
\begin{eqnarray}\label{eq:change}
 Q(\Phi) &=& \int_{\A} \omega_\Phi(\alpha) \d\mu(\alpha) = \int_{\A} q_\alpha(\Phi(\alpha)) \d\mu(\alpha)
                                          =  \int_\A\left(\int_\R \lambda \;\d\nu_\Phi^\alpha(\lambda) \right)\;\d\mu(\alpha) \\
                                        & = & \int_\A \scalar{|T_\alpha |^{1/2}\Phi(\alpha)}{(\operatorname{sgn}T_\alpha) |T_\alpha |^{1/2}\Phi(\alpha)}_\alpha
                                             \;\d\mu(\alpha)  \, . \nonumber
\end{eqnarray}

Moreover, from \cite{nussbaum64} we have $T = \int_\A^\oplus T_\alpha\d\mu(\alpha)$ and $|T|^{1/2}= \int_\A^\oplus |T_\alpha|^{1/2}\d\mu(\alpha)$ and, therefore,
for $\Phi\in\D(|T|^{1/2})$ it is clear that
$$\scalar{|T|^{1/2}\Phi}{|T|^{1/2}\Phi} = \int_\A\scalar{|T_\alpha|^{1/2}\Phi(\alpha)}{|T_\alpha|^{1/2}\Phi(\alpha)}_\alpha\d\mu(\alpha)<\infty,$$
which implies by dominated convergence that
$$\scalar{|T|^{1/2}\Phi}{(\operatorname{sgn}T)|T|^{1/2}\Phi} = \int_\A\scalar{|T_\alpha|^{1/2}\Phi(\alpha)}{(\operatorname{sgn}T_\alpha)|T_\alpha|^{1/2}\Phi(\alpha)}_\alpha\d\mu(\alpha).$$
Since $\bD = \overline{\D}^{\normm{\cdot}_E}=\D(|T|^{1/2})$ the quadratic form $(\bQ,\bD)$ extends $(Q,\D)$.
\end{proof}

An immediate application of Theorem~\ref{thm:mainthm1} is provided in the particular instance that the quadratic forms $q_\alpha$ are semibounded since by Corollary~\ref{cor:Katorepisstrongrep} this implies \ref{H3} (see also Theorem~\ref{thm:representation-2}). This special situation appears in the examples considered in Section~\ref{sec:examples}.

\begin{corollary}\label{cor:mainthm}
    Let $Q\colon\D\to\R$ be a densely defined quadratic form with domain $\D$ satisfying \ref{H1}, \ref{H2} and such that the quadratic forms $q_\alpha$ associated to $Q$ are semibounded on $P_\alpha\D$.   Assume that $\D \subset \D(| T |^{1/2})$. Then $Q$ is strongly representable by the self-adjoint operator $T$ defined in Proposition \ref{pro:integrateE}.
\end{corollary}

Theorem~\ref{thm:mainthm1} solves the problem of determining when an orthogonally additive quadratic form defined on a direct integral Hilbert space is strongly representable (provided that conditions \ref{H1}, \ref{H2} and \ref{H3} are satisfied). However it could be extremely hard in practical applications to compute explicitly the domain $\D (|T|^{1/2})$, so that the important condition characterising strong representability, $\D \subset \D (|T|^{1/2})$, can be tested. We conclude this section by  providing an alternative description of the domain $\D (|T|^{1/2})$ that could be helpful in numerical analysis and other applications as shown in the next section, e.g., Theorem~\ref{teo:postion-op} or Theorem~\ref{thm:ginvariantoa}. The main idea will be to define a natural dense domain $\Df\subset\H$ which imposes boundedness conditions both on
$\A$ and the support of $E(\cdot)$ and that plays the role of a core for the not-semibounded quadratic form $(Q,\D)$.
This set will be shown to be dense with respect to different norms in $\bD$ and $\D(T)$.
Moreover, it turns out to be stable under the action of the projections compatible with the direct integral decomposition of the Hilbert space, i.e.,  $P_\Delta\Df \subset \Df$.
It is therefore a natural space to use for approximations for numerical methods, as in \cite{Lop17}, or
to compute the spectrum and eigenvectors associated to $\bQ$ or $T$.

\begin{definition}\label{def:dfin}
Let $Q$ be a quadratic form densely defined on $\D$ which is dense in the direct integral Hilbert space $\H$ and
satisfying \ref{H1}, \ref{H2} and \ref{H3}.
Let $\{E^{\alpha}(\cdot)\}_{\alpha\in\A}$ be the family of resolutions of the identity given in Proposition~\ref{pro:integrateE} and consider the projections
$P_\Delta$, $\Delta\in\Sigma(\A)$, and $E(\sigma)$, $\sigma\subset\R$. Define the following subspace of $\H$ by
$$\Df := \left\{ E(\sigma) P_\Delta \Phi \in \H \;\Bigl|
                \; \Phi\in\D\;, \; \mu(\Delta)<\infty\quad\mathrm{and}\quad\sigma\subset\R \text{ compact}
       \right\}\,.$$
\end{definition}

Note that for any $\Delta\in\Sigma(\A)$ and any Borel set $\sigma\subset\R$ the projections $P_\Delta$ and $E(\sigma)$ commute.

The definition of the set $\Df$ is justified by the fact that for all $\Phi \in \Df$ one has that
\begin{equation}\label{lem:dffin}
    \int_\A \int_\mathbb{R} |\lambda| \d\nu^{\alpha}_{\Phi}(\lambda) \d\mu(\alpha) < \infty\,.
\end{equation}

\begin{definition}\label{def:extensionD}
    Let $Q$ be a quadratic form densely defined on $\D$ which is dense in the direct integral Hilbert space $\H$ and
    satisfying \ref{H1}, \ref{H2} and \ref{H3}. We define a norm on $\Df$ by
    $$\normm{\Phi}^2 = \norm{\Phi}^2 + \int_\A\int_\R |\lambda| \d\nu^{\alpha}_{\Phi}(\lambda) \d \mu(\alpha)\,, \quad\Phi\in\Df \;. $$
    We denote again by $\normm{\cdot}$ the norm on the closure $\overline{\Df}^{\normm{\cdot}}$.
\end{definition}

\begin{theorem}\label{thm:mainthm2}
Let $Q$ be a quadratic form with domain $\D$ which is dense in the direct integral Hilbert space $\H$, 
satisfying \ref{H1}, \ref{H2}, \ref{H3} and such that $\D\subset \overline{\Df}^{\normm{\cdot}}$. Let $T$ be the self-adjoint operator associated to the resolution of the identity $E(\cdot)$ obtained in Proposition~\ref{pro:integrateE}. Then $(Q,\D)$ is strongly representable by $T$ and the extension coincides with $(\bQ,\bD)$ given in Theorem~\ref{thm:mainthm1}.

\end{theorem}
\begin{proof}
We begin showing that the norms $\normm{\cdot}$ and $\normm{\cdot}_E$ coincide on $\H$ (allowing the norm to take the value infinity).
In fact, for any Borel set $\sigma\subset\R$ and $\Phi\in\H$ note that
\[
 \nu_{\Phi}(\sigma)=\langle\Phi,E(\sigma)\Phi \rangle = \int_\A\langle \Phi(\alpha),E^\alpha(\sigma)\Phi(\alpha) \rangle \d\mu(\alpha)
                                                      = \int_\A \nu_\Phi^\alpha(\sigma) \d\mu(\alpha)\;.
\]
Therefore $\int_\R s(\lambda)\d\nu_\Phi(\lambda)=\int_\A\int_\R s(\lambda) \d\nu_\Phi^\alpha(\lambda) \d\mu(\alpha)$ for any simple function $s$. Choosing a lower approximation of
the function $f(\lambda)=|\lambda|$ by simple functions we conclude by monotone convergence that
\[
 \int_{\R}|\lambda|\d\nu_{\Phi}(\lambda)  = \int_\A\int_{\R}|\lambda|\d\nu^\alpha_{\Phi}(\lambda)\d\mu(\alpha),
\]
showing that the norms coincide.

Clearly $\Df \subset \D(|T|^{1/2})$. Hence $\D\subset \overline{\Df}^{\normm{\cdot}} \subset \D(|T|^{1/2})$ and the result follows from Theorem~\ref{thm:mainthm1}.

\end{proof}

\section{Examples and applications}\label{sec:examples}
In this section we will present some examples that illustrate the structures needed
for the representation theorem for non-semibounded Hermitian quadratic forms stated before.

\subsection{The position operator in Quantum Mechanics}\label{subsec:position}

The first example of a multiplication operator is in a sense prototypical, because
by the spectral theorem, any self-adjoint operator $T$ representing the quadratic form will
have a decomposition $T=\int_\R\lambda\d E(\lambda)$, for a uniquely determined resolution
of the identity $E(\cdot)$.
The main idea here is to make contact with the structures introduced in the preceding section
by considering a coarsening of $\R$ labeled by the integers $\Z$, e.g., one can consider a uniform partition
$\R=\sqcup_{k\in\Z} I_k$ with $I_k=[k, k+1)$.
Then the Hilbert space of square integrable functions on $\R$ with the Lebesgue measure has a natural decomposition
\[
\H = L^2(\R)\cong \mathop{\bigoplus}\limits_{k\in\Z}L^2(I_k) \;.
\]
Define the domain $\D$
as the space of piecewise continuous functions with compact support,
which is dense in $L^2(\R)$, and consider finally the Hermitian quadratic form on $\D$ defined by
\begin{equation}\label{eq:positionQF}
Q(\Phi,\Psi) := \int_{\R} x\;\overline{\Phi(x)}\,\Psi(x) \,\d x
              =\mathop{\sum}\limits_{k\in\Z}\int_{I_k} x\;\overline{\Phi(x)}\,\Psi(x)\, \d x \;,\quad \Phi,\Psi\in\D\;.
\end{equation}
As a measure space we consider $\A=\Z$ with $\sigma$-algebra $\Sigma(\A)=\mathcal{P}(\Z)$ given by all subsets of $\Z$ and the counting
measure $\mu\colon\mathcal{P}(\Z)\to\N_0$. Putting $\H^k = L^2(I_k, \d x)$ for $k \in \Z$ we have
$$\H \cong \int^{\oplus}_\Z \H^k \d \mu(k) = \bigoplus_{k\in\Z} \H^k\,.$$
For any $\Phi = \int_{\mathbb{Z}}^\oplus \Phi^k\d\mu(k)$ its norm satisfies
\[
\|\Phi\|^2=\int_\Z \,\norm{\Phi^k}^2_k\,\d\mu(k) = \sum_{k\in\Z} \int_{I_k}|\,\Phi^k(x)|^2\,\d x < \infty.
\]
The projection operator $P_k\colon \H \to \H^k$ can be identified with the multiplication operator by the characteristic function of the interval $I_k$, i.e., $(P_k \Phi)(x) = \mathbf{1}_{I_k}(x)\Phi(x)$. For each $k\in\Z$ the quadratic form defined in
Definition~\ref{def:qalpha} is given simply by
\[
 q_k(\Phi^ k)=Q(P_k\Phi)=\int_{I_k} x |\Phi^k(x)|^2\, \d x\;.
\]

\begin{proposition}\label{pro:positionCOA}
The quadratic form $Q$ defined in Eq.~(\ref{eq:positionQF}) on the dense domain $\D$ of piecewise continuous functions
with compact support is closable and countably orthogonally additive.
\end{proposition}

\begin{proof}
Note that the multiplication operator is symmetric on $\D$, hence by Corollary~\ref{cor:symmetric} the quadratic form $Q$ is
closable. According to Definition~\ref{def:countableoa} we check that the domain $\D$ is stable under the projections $P_\Delta$ for
any $\Delta\subset\Z$. Since $\Phi\in\D$ has compact support it will intersect with only finitely many $I_k$, $k\in\Delta$. Therefore,
\[
 (P_\Delta\Phi)(x)=\sum_{k\in\Delta} (\mathbf{1}_{I_k}\Phi)(x)\qquad \mathrm{(finite~sum)}
\]
is also piecewise continuous with compact support, hence $P_\Delta\Phi\in\D$. Finally, for any $\Delta\in\mathcal{P}(\Z)$ consider a partition $\Delta=\sqcup_{j\in J}\Delta_j\in\mathcal{P}(\Z)$. By definition of the quadratic form we have for any $\Phi\in\D$
    \begin{align*}
    Q(P_\Delta\Phi) & = \mathop{\sum}\limits_{k\in\Delta}\int_{I_k} x\;|\Phi(x)|^2 \, \d x
              \;=\;\mathop{\sum}\limits_{j\in J} \mathop{\sum}\limits_{k_j\in\Delta_j}\int_{I_{k_j}} x \;|\Phi(x)|^2\,  \d x \\
            & = \mathop{\sum}\limits_{j\in J}  Q(P_{\Delta_j}\Phi) \;,
    \end{align*}
which shows countable orthogonal additivity.
\end{proof}

\begin{remark}
 This example already shows that the direct integral structure of the Hilbert space is highly non-unique. An alternative decomposition can be given by
 choosing the partition $I_+:=[0,\infty)$ and $I_-:=(-\infty, 0)$ (see also Lemma~2.5 in \cite{mcintosh3} where a similar decomposition is given).
\end{remark}
In the next result we show that the example satisfies all hypothesis of Theorem~\ref{thm:mainthm1} and hence we can
apply the representation theorem.

\begin{theorem}\label{teo:postion-op}
The quadratic form $Q$ defined in Eq.~(\ref{eq:positionQF}) on the dense domain $\D$ of piecewise continuous functions
with compact support satisfies $\D\subset\D(|T|^{1/2})$ as well as \ref{H1}, \ref{H2} and \ref{H3}.
It is strongly representable by the self-adjoint multiplication operator $(T\Phi)(x)=x\,\Phi(x)$ with domain
$\D(T)=\{\Phi\in\H \mid \int_\R x^2 | \phi(x)|^2<\infty \}$.
\end{theorem}
\begin{proof}
By Proposition~\ref{pro:positionCOA} the quadratic form $Q$ on $\D$ satisfies \ref{H1}. Moreover, by construction,
the measure space $(\Z,\mathcal{P}(\Z),\mu)$
is point supported, hence \ref{H2} is also satisfied. To check \ref{H3} consider for any $k\in\Z$ the quadratic form
$q_k\colon P_k\D\to\R$ (cf.\ Definition~\ref{def:qalpha}). Since
\[
| q_k(\Phi^k)| = |Q(P_k\Phi)| = \Big|\int_{I_k} x\;|\Phi(x)|^2 \, \d x\Big|\leq \max\{|k|,|k+1|\}\; \|\Phi^k\|^2
\]
it follows that $q_k$ is bounded so by Corollary~\ref{cor:Katorepisstrongrep} condition \ref{H3} follows. Finally, since any $\Phi\in\D$ is piecewise continuous and has compact
support it is bounded and there exist a finite $\Delta\subset\Z$ and a compact interval $\sigma\subset\R$ such that
$P_\Delta E(\sigma) \Phi=\Phi$, hence $\D\subset\Df$  (cf.\ Definition~\ref{def:dfin}) and therefore, by Theorem~\ref{thm:mainthm2}, $\D\subset\D(|T|^{1/2})$. The statement on the strong representability of $Q$ follows from the Theorem~\ref{thm:mainthm1}.

It remains to show that the representing operator coincides with the multiplication operator $T$. The representing operator
of the extension $\bQ$ on $\bD$ is the unique self-adjoint operator associated with the resolution of the
identity $E$ constructed in Proposition~\ref{pro:integrateE}. In our example, for any $k\in\Z$ the bounded quadratic forms
$q_k$ are represented by the multiplication operator by $x$ on the interval $I_k$, hence the associated resolution of the identity
is given by multiplication with the characteristic function relative to the interval $I_k$, i.e.,
$E^k(\sigma)\Phi^k=\mathbf{1}_{\sigma\cap I_k}\Phi^k(x)$, $\sigma\subset\R$. By Proposition~\ref{pro:integrateE} we have
\[
 E(\sigma)\cong\bigoplus_{k\in\Z} E^k(\sigma)=\mathbf{1}_\sigma\;.
\]
which is the resolution of the identity of $T$.
\end{proof}

We can now invoke the spectral theorem to prove a certain converse implication to Theorem~\ref{thm:mainthm1}:

\begin{theorem}\label{thm:reverseimplication}
    Let $\D\subset\H$ be dense and $Q\colon\D\to\R$ be a closable quadratic form and strongly representable by a self-adjoint operator $T$ with resolution of the identity $E(\cdot)$. Assume that $Q=\bQ$ (and $\D =\bD$). Then there exists a measure space $(\mathbb{Z}, \mathcal{P}(\mathbb{Z}), \mu)$ and a decomposition of $\H$ as a direct integral such that $Q$ satisfies \ref{H1}, \ref{H2} and \ref{H3}.
\end{theorem}

\begin{proof}
    We will partition $\mathbb{R}$ as in the example, i.e., $\R=\sqcup_{k\in\mathbb{Z}} I_k,$ with $I_k = [k, k+1)$. Let $\H^k = \operatorname{ran}E(I_k)$. From the properties of the resolution of the identity we have that
    $$
        \H = \bigoplus_{k\in\mathbb{Z}}\H^k = \int^\oplus_{\mathbb{Z}}\H^k \d\mu,
    $$
    where $\mu$ is the counting measure and we take as $\sigma$-algebra $\Sigma(\mathbb{Z}) = \mathcal{P}(\mathbb{Z})$. This proofs \ref{H2}. To show condition \ref{H1} it
    is enough to show orthogonal additivity (cf.\ Theorem~\ref{thm:oa=coa}). Hence by Definition~\ref{def:oa} we have to show stability of the domain, $\Sigma$-boundedness and additivity.
    Since by strong representability $\D(T)\subset\bD=\D\subset\D(|T|^{1/2})$ it is enough to notice that for any $\Delta\in\mathcal{P(\mathbb{Z})}$ such that $\mu(\Delta)<\infty$ and any $\xi\in\bD$ one has that $P_\Delta\xi\in\D(T)$ and stability follows by monotone convergence.
    For any $\Phi\in\bD$ we have that
    $$
        |Q(P_\Delta\Phi)| = \left| \int_\Delta \lambda\d\nu_{\Phi}(\lambda) \right|\leq  \int_\Delta |\lambda|\d\nu_{\Phi}(\lambda) \leq  \int_\R |\lambda|\d\nu_{\Phi}(\lambda) := M_\Phi < \infty.
    $$
    Since $M_\Phi$ does not depend on $\Delta$, this shows $\Sigma$-boundedness.     Finally, for any $\Delta\in\mathcal{P}(\Z)$ consider a finite partition $\Delta=\sqcup_{j=1}^N\Delta_j\in\mathcal{P}(\Z)$. For any $\Phi\in\bD$ we have
    $$
        Q(P_\Delta\Phi) =\lim_{n\to\infty} \sum_{%
            {{k\in\Delta} \atop {|k| < n}}} \int_{I_k} \lambda \d \nu_{\Phi} (\lambda)
            =\sum_{j=1}^N\lim_{n\to\infty} \sum_{%
            {{k_j \in \Delta_j} \atop {|k_j| < n}}} \int_{I_{k_j}} \lambda \d \nu_{\Phi} (\lambda)
            = \sum_{j=1}^N Q(P_{\Delta_j}\Phi),
    $$
    where we have used dominated convergence. Finally, note that the quadratic forms $q_k\colon P_k\D\to \R$ are bounded and, therefore, \ref{H3} follows from
    Corollary~\ref{cor:Katorepisstrongrep}.
\end{proof}

Notice that, in the case that $Q$ is strongly representable, but one has that $Q \neq \bQ$
it is not true in general that the stability of the domain $\D$ holds.

\begin{corollary}\label{cor:reverseimplication}
    Let $T$ be a self-adjoint operator with resolution of the identity $E(\cdot)$ and domain
    $$
        \D(T) = \left\{\Phi \in \H \mid \int_\R \lambda^2 \d\nu_\Phi(\lambda) < \infty\right\}.
    $$
    Let $Q$ be the quadratic form densely defined on $\D(T)$ given by
    $$
        Q(\Phi) = \scalar{\Phi}{T\Phi}.
    $$
    Then $Q$ is closable, strongly representable and there exists a measure space $(\Z,\mathcal{P}(\Z),\mu)$
and a decomposition of $\H$ as a direct integral such that $Q$ satisfies \ref{H1}, \ref{H2} and \ref{H3}.
\end{corollary}

\begin{proof}
    By Theorem~\ref{thm:reprimpliesclosable} the quadratic form $Q$ is closable and putting $\bD := \overline{\D(T)}^{\normm{\cdot}_E}$
    it is clear that $\D(T)=\D\subset\bD$ and $\D=\D(T)\subset \D(|T|^{1/2})$.
    As in Theorem~\ref{thm:mainthm1} on can show that $Q$ can be extended to a form \hbox{$\bQ(\Phi) = \int_{\R}\lambda\d\nu_{\Phi}(\lambda)$} defined on $\bD$ and hence $Q$ is strongly representable.
    Using an analogous coarsening as in Theorem~\ref{thm:reverseimplication} but now on the measure space associated to the spectral resolution of the operator $T$ it is straightforward to show that $\bQ$ and therefore $Q$ satisfies \ref{H2} and \ref{H3}. We still have to show orthogonal additivity of $Q$. The properties of $\Sigma$-boundedness and additivity of $Q$ follow from those of $\bQ$, which hold by the Theorem~\ref{thm:reverseimplication}. Finally, we need to show stability of domain: let $\Phi\in\D(T)$ and $\Delta\in\mathcal{P}(\Z)$. Since $P_{\Delta}\Phi = E(\sqcup_{k\in\Delta}I_k)\Phi$ we have that
    $$\int_{\R}\lambda^2 \d\nu_{P_{\Delta}\Phi}(\lambda) = \int_{\sqcup_{k\in\Delta}I_k}\lambda^2 \d\nu_{\Phi}(\lambda) \leq \int_{\R}\lambda^2 \d\nu_{\Phi}(\lambda) < \infty,$$
    which concludes the proof.
\end{proof}

\begin{remark}
    We have seen in Theorem~\ref{thm:mainthm1} that conditions \textbf{H1}, \textbf{H2} and \textbf{H3} give a sufficient condition for a (non-semibounded) quadratic form $(Q,\D)$ to be strongly representable by a self-adjoint operator $T$. The extension of the quadratic form $(\bQ,\bD)$ has then an integral representation as given in Definition~\ref{def:strongrepresentability}. Now Theorem~\ref{thm:reverseimplication} and Corollary~\ref{cor:reverseimplication} show that  if a quadratic form is strongly representable, then the integral representation of $T=\int_\R\lambda \d E(\lambda)$ guaranteed by the spectral theorem naturally gives an integral representation of the extended quadratic form $(\bQ,\bD)$. One can then construct a decomposition of the Hilbert space and verify conditions \textbf{H1}, \textbf{H2} and \textbf{H3} in this case.
\end{remark}

\subsection{Unitary representation of groups}\label{subsec:groups}

The theory of unitary representations of groups on Hilbert spaces is an important situation in which direct
integral decompositions appear naturally. In this framework, the notion of commutant of a family of bounded operators is particularly useful.
If $\mathcal{S}$ is a self-adjoint subset of $\mathcal{L}(\H)$ (the set of bounded linear operators on the Hilbert space $\H$),
we denote by $\mathcal{S}'$ the commutant of $\mathcal{S}$
in $\mathcal{L}(\H)$, i.e., the set of all operators in
$\L(\H)$ commuting with all elements in $\mathcal{S}$. It is a consequence of von Neumann's bicommutant theorem that
$\mathcal{S}'$ is a von Neumann algebra and that the corresponding
bicommutant $\mathcal{S}'':=(\mathcal{S}')'$ is the smallest von Neumann algebra
containing $\mathcal{S}$.
We refer, e.g., to Sections~4.6 and 5.2 of \cite{Pedersen89} for additional motivation and proofs.
Let $V\colon G\to \mathcal{U}(\H)$ be a unitary representation; consider the following von Neumann algebras associated with
this representation
\[
 \mathcal{V}:=\{V(g) \mid g\in G\}''\quad\mathrm{and}\quad
 \mathcal{V}'=\{M\in\mathcal{L}(\H)\mid  V(g)M=MV(g)\;,\;g\in G\} \;.
\]
It is shown in Section~2.4 of \cite{mackey:76} that any unitary representation of
a separable locally compact group $G$ is a direct integral of primary representations. A representation is primary if $\mathcal{V}\cap\mathcal{V}'=\mathbb{C}\1$.  If, in addition, $G$ is of type~I
(e.g., if any subrepresentation of $V$ contains an irreducible subrepresentation), then the direct integral
decomposition is essentially unique and
one can take as measure space $\A=\widehat{G}$, the dual of $G$, i.e., the set of all unitary equivalence classes of
irreducible representations.

To combine these results on unitary representations with our analysis of quadratic forms
we have to adapt the notion of $G$-invariance of quadratic forms developed in Section~4 of
\cite{Ib14b} to the underlying structure of direct integrals.

\begin{definition}\label{def:ginv}
    Let $Q$ be a closable, Hermitian quadratic form on $\D$ which is dense in the direct integral Hilbert space $\H$.
    Let $G$ be a separable locally compact group
    with a unitary representation $V$ on the Hilbert space $\H$. We will say that the quadratic form is \textbf{$G$-invariant} if
    \begin{enumerate}
        \item {\em Stability of domain:} For $\Delta\in\Sigma(\A)$ one has that $P_\Delta\D \subset \D$ \label{item:ginv1}
        \item {\em Boundedness on finite measure sets:}
            For any $\Delta \in \Sigma(\A)$ with $\mu(\Delta)<\infty$ there exists a $M_\Delta>0$ such that
            $$|Q(P_\Delta\Phi)| \leq M_{\Delta} \norm{P_\Delta \Phi}^2 \text{ for all } \Phi\in\D.$$ \label{item:ginv2}
        \item {\em $\Sigma$-boundedness:} For all $\Phi\in\D$ there exists $M_\Phi >0$ such that
            $$|Q(P_\Delta\Phi)| < M_\Phi \text{ for all }\Delta \in \Sigma(\A).$$ \label{item:ginv3}
        \item {\em Invariance of the domain and the form:} $V(g)\D=\D$ and $Q(V(g)\Phi) = Q(\Phi)$ for all $g\in G$ and $\Phi \in \D$. \label{item:ginv4}
    \end{enumerate}
\end{definition}

Note that condition (\ref{item:ginv2}) is only required for sets $\Delta$ with finite measure, hence
it {\em does not} imply that the quadratic form is bounded. Moreover, by polarisation (cf.\ Eq.~(\ref{eq:quadratic-polarization})),
invariance already implies that the sesquilinear form is also invariant, i.e., $Q(V(g)\Phi,V(g)\Psi)=Q(\Phi,\Psi)$, $g\in G$,
$\Phi,\Psi\in\D$.

For the next result we need to recall the notion of disjoint representations. Two unitary representations
$U,V\colon G\to \mathcal{U}(\H)$ are called disjoint if the set of intertwining operators is trivial, i.e., if
\[
 (U,V):=\{ M\in\mathcal{L}(\H)\mid  MU(g)=V(g)M\;,\; g\in G \} =\{0\}
 \;.
\]

For the next result we need to be more specific of how the measure space appears in the decomposition of $V$. We will
follow here von Neumann's central decomposition. Denote by
\[
 \mathcal{Z}=\mathcal{V}\cap\mathcal{V}'
\]
the centre of the von Neumann algebra $\mathcal{V}$. Then there is a separable Hausdorff space
$\A$ and a regular Borel measure $\mu$ on $\A$ such that $\mathcal{Z}\cong L^\infty(\A,\mu)$ (cf.\ \cite[Theorem 14.8.14]{wallach:92}).

\begin{lemma}\label{lem:disjoint}
 Let $V\colon G\to \mathcal{U}(\H)$ be a unitary representation and consider its central decomposition mentioned above (cf.\ \cite[Theorem 2.9]{mackey:76} or \cite[Corollary 14.9.5]{wallach:92}), i.e.,
 \begin{equation}\label{eq:decompV}
        V = \int^\oplus_\A V^\alpha \d\mu(\alpha)\;.
 \end{equation}
Then, for any pair $\Delta_1,\Delta_2\in\Sigma(\A)$ with
$\Delta_1\cap\Delta_2=\emptyset$, we have
\[
 (V_{\Delta_1},V_{\Delta_2})=\{0\} \;,
\]
where $V_{\Delta_i}= \int^\oplus_{\Delta_i} V^\alpha \d\mu(\alpha)$, $i=1,2$, denote the corresponding subrepresentations.
\end{lemma}
\begin{proof}
Let $\Delta_i$, $i=1,2$, be as above. Then the corresponding projections $P_{\Delta_i}$ can be identified with characteristic
functions in $L^\infty(\A,\mu)\cong\mathcal{Z}$, hence they are central. Therefore by Theorem~1.3 in \cite[Chapter~1]{mackey:76}
it follows  that  $(V_{\Delta_1},V_{\Delta_1^c})=\{0\}$, i.e., they are disjoint. Since $V_{\Delta_2}$ is a subrepresentation of  $V_{\Delta_1^c}$ the statement follows.
\end{proof}

In \cite[Theorem 4.2]{Ib14b} we relate the $G$-invariance of closed, semibounded quadratic forms with that of a representing operator $T$. In \cite[Corollary 5.10]{Ib14b} we show that if an operator $T$ is unbounded and $G$-invariant for a representation $V$ then the representation cannot be a finite direct sum of finitely many irreducible representations.
In the next result we consider a direct integral decomposition of $V$, where the ``pieces'' of the decomposition
are mutually disjoint.

\begin{theorem}\label{thm:ginvariantoa}
Let $G$ be a separable locally compact group and let $V$ be a unitary representation acting on a Hilbert
space $\H$ and consider its central decomposition like in Eq.~\eqref{eq:decompV} with $\mu$ a $\sigma$-finite positive measure.
Let $Q$ be a closable Hermitian quadratic form, densely defined on $\D_0\subset\H$, where for any $\Phi\in\D_0$ there
exist a $\Delta\in\Sigma(\A)$ with $\mu(\Delta)<\infty$ and $P_\Delta\Phi=\Phi$.
\begin{itemize}
 \item[(i)] If $Q$ is $G$-invariant, then $Q$ is countably orthogonally additive.
 \item[(ii)] If, in addition, the measure $\mu$ is point supported on $(\A,\Sigma(\A))$, then the $G$-invariant quadratic form
 $Q$ is strongly representable.
\end{itemize}

\end{theorem}

\begin{proof}
(i) By Theorem~\ref{thm:oa=coa} it is enough to prove orthogonal additivity.
Moreover, we only have to show the additivity property (\ref{item:oa3}) of Definition~\ref{def:oa}
since conditions (\ref{item:oa1}) and (\ref{item:oa2}) follow directly from the definition of $G$-invariance.
It is enough to show that for any pair $\Delta_1, \Delta_2 \in \Sigma(\A)$
with $\Delta_1 \cap \Delta_2 = \emptyset$ and $\mu(\Delta_i) <\infty$, $i=1,2$, one has that
$$Q(P_{\Delta_1}\Phi,P_{\Delta_2}\Psi) = 0, \; \forall \,\Phi, \Psi \in \D_0.$$

We will use next the $G$-invariance of the quadratic form $Q$ to define an intertwiner between
the subrepresentations $V_{\Delta_i}:=V P_{\Delta_i}$ acting on the Hilbert spaces $\H_i:=P_{\Delta_i}\H$, $i=1,2$.
Consider for a fixed $\Psi\in\D_0$ the linear map
\[
 Q_{\Psi,12}\colon P_{\Delta_2}\D\subset\H_2\to\CC\;,\quad Q_{\Psi,12}(P_{\Delta_2}\Phi):=Q(P_{\Delta_1}\Psi, P_{\Delta_2}\Phi)\;.
\]
By polarisation and since $\mu(\Delta_1\sqcup\Delta_2)=\mu(\Delta_1)+\mu(\Delta_2)$ we have from Definition~\ref{def:ginv}~(ii)
that $Q_{\Psi,12}$ is bounded by $M_{\Delta_1\sqcup\Delta_2}\norm{P_{\Delta_1}\Psi}$ and, therefore, can be extended to a bounded linear functional
$Q_{\Psi,12}\colon \H_2\to\CC $. By Riesz's representation theorem there exists a $\chi_2\in\H_2$ such that
\[
  Q_{\Psi,12}(\Phi_2)=\langle \chi_2,\Phi_2\rangle\;,\quad \Phi_2\in\H_2\;.
\]
Consider now the map
\[
C\colon P_{\Delta_1}\D_0\to\H_2\;,\quad C\left(P_{\Delta_1}\Psi\right):=\chi_2\;,
\]
which is linear since $Q$ is sesquilinear and bounded by $M_{\Delta_1\sqcup\Delta_2}$. Moreover, since the projections \hbox{$P_{\Delta_i}$, $i=1,2$}, reduce the representation
$V$ we have for any $g\in G$, $\Psi,\Phi\in\D_0$,
\begin{eqnarray*}
  \big\langle C\big(P_{\Delta_1}V(g)\Psi\big)\,,\,P_{\Delta_2}\Phi\big\rangle
                  &=& Q\big( P_{\Delta_1}V(g)\Psi \,,\,P_{\Delta_2}\Phi\big)
                 \;=\; Q\big( P_{\Delta_1}\Psi \,,\,V(g)^* P_{\Delta_2}\Phi\big)\\
                 &=&  \big\langle V(g) C\big(P_{\Delta_1}\Psi\big)\,,\,P_{\Delta_2}\Phi\big\rangle\;,
\end{eqnarray*}
hence
\begin{equation}
   \big\langle C(P_{\Delta_1}V(g)\Psi)-V(g)C(P_{\Delta_1}\Psi)\,,\,P_{\Delta_2}\Phi\big\rangle  = 0\\
\end{equation}
and, from the density result in Lemma~\ref{lem:denseD}, we conclude that
$$ CV_{\Delta_1}(g) = V_{\Delta_2}(g)C,$$
i.e., $C\in \big(V_{\Delta_1},V_{\Delta_2}\big)$ and intertwines both representations. By Lemma~\ref{lem:disjoint} we have that $V_{\Delta_1}$ and $,V_{\Delta_2}$ are disjoint
so that $C=0$ and therefore $Q(P_{\Delta_1}\Phi,P_{\Delta_2}\Psi) = 0$.

(ii) The first part of the theorem shows that condition \ref{H1} holds and, by assumption on the measure space, we have also \ref{H2}.
Moreover, \ref{H3} is a direct consequence of the boundedness assumption of finite measure sets in Definition~\ref{def:ginv}~(\ref{item:ginv2}) and Corollary~\ref{cor:Katorepisstrongrep}.
By Theorem~\ref{thm:mainthm1} we have still to show that $\D_0\subset\D(|T|^{1/2})$. By Theorem~\ref{thm:mainthm2} it is enough to show that $\D_0 = \Df$, cf.\ Definition~\ref{def:dfin}. This follows by the boundedness on finite measure sets and the fact that for all $\Phi\in\D_0$ we have that $P_\Delta\Phi = \Phi$ for some $\Delta\in\Sigma(\A)$ with $\mu(\Delta)<\infty$.
\end{proof}

\begin{example}\label{ex:compact-group} {\bf (Compact groups and Peter-Weyl theorem)}
 Let $G$ be a compact group and consider the Hilbert space $\H=L^2(G)$ with the corresponding Haar measure.
 Consider the left regular representation of $G$ on $\H$: $(L(g_0)\varphi):=\varphi(g_0^{-1}g)$, $g_0\in G$.
 By Peter-Weyl's theorem (see, e.g., \cite[\S27.49]{hewitt-ross-2}) we have
 \[
  L^2(G)=  {\bigoplus_{\a \in \widehat{G}} n(\a) \H^\a}
  \quad\mathrm{and}\quad
  L(g) = {\bigoplus_{\a \in \widehat{G}} n(\a) L^\a}
 \;,
 \]
 where $\mathcal{H}^\alpha$ is the support space of the irreducible representation labeled by $\alpha$ and the multiplicity $n(\alpha)$ satisfies $n(\a)=\mathrm{dim} \H^\a<\infty$. Therefore, we choose as a
 discrete space $\A:=\widehat{G}$ with the counting measure $\mu$ for the decomposition of $L$.
 \begin{lemma}\label{lem:cond}
  Let $Q$ be a quadratic form on $\H=L^2(G)$ on a dense domain $\D_0$ with the property that if $\Phi\in\D_0$, then there
  is a $\Delta\subset \widehat{G}$ with $\mu(\Delta)<\infty$ and $P_\Delta\Phi=\Phi$. Then,
  \begin{itemize}
   \item[(i)] for any, $\Delta\subset \widehat{G}$ with $\mu(\Delta)<\infty$ we have $P_\Delta\D_0=P_\Delta\H$ and $P_{\Delta}\D_0\subset\D_0$;
              the quadratic form $Q$ restricted to $P_\Delta\H$ is bounded;
   \item[(ii)] The domain $\D_0$ is invariant;
   \item[(iii)] The quadratic form $Q$ is $\Sigma$-bounded, i.e., for any $\Delta\subset \widehat{G}$ we have
   $|Q(P_\Delta\Phi)|< M_\Phi$,\quad \hbox{$\Phi\in\D_0$}.
  \end{itemize}
 \end{lemma}
 \begin{proof}
  (i) Since $\D_0$ is dense and any $\H^\a$ is finite dimensional and appears with finite multiplicity we conclude that for any
   $\Delta\subset \widehat{G}$ with $\mu(\Delta)<\infty$ we have $P_\Delta\D_0=P_\Delta\H$. Moreover, the restriction of $Q$ to
   the finite dimensional Hilbert space $P_\Delta\H$ is finite and, by the defining property of
   $\D_0$, we conclude that it is also stable, i.e., $P_{\Delta}\D_0\subset\D_0$.

  (ii) Since any $P_\Delta$, $\Delta\subset\widehat{G}$, reduces $L$ we have from the definition of $\D_0$ that $L(g)\D_0=\D_0$,
  $g\in G$.

  (iii) Finally, to prove $\Sigma$-boundedness let $\Phi\in\D_0$. Then, there exists $\Delta_\Phi\subset \widehat{G}$ with $\mu(\Delta_\Phi)<\infty$ such that $\Phi = P_{\Delta_\Phi}\Phi$. Take $\Delta \subset \widehat{G}$. Then we have that
    $$|Q(P_\Delta\Phi)| = |Q(P_{\Delta_\Phi} P_\Delta\Phi)| \leq M_{\Delta_\Phi}\norm{P_{\Delta_\Phi}P_\Delta \Phi}^2 \leq M_{\Delta_\Phi}\norm{\Phi}^2=: M_{\Phi}\,,$$
    where we have used that (\ref{item:ginv2}) of Definition~\ref{def:ginv} holds.
 \end{proof}

 \begin{proposition}
  Let $G$ be a compact metrisable group and
  $Q$ be a closable quadratic form on $\H=L^2(G)$ on a dense domain $\D_0$ with the property that if $\Phi\in\D_0$ then there
  is a $\Delta\subset \widehat{G}$ with $\mu(\Delta)<\infty$ and $P_\Delta\Phi=\Phi$. If $Q$ is invariant under the left regular representation $L$, i.e., $Q(L(g)\Phi)=Q(\Phi)$, $g\in G$, then $Q$ is strongly representable.
 \end{proposition}
 \begin{proof}
  Note that by Lemma~\ref{lem:cond} the $L$-invariance of $Q$ already implies that $Q$ is $G$-invariant
  (recall Definition~\ref{def:ginv}) and, by Theorem~\ref{thm:ginvariantoa}~(i),
  we conclude that $Q$ is countably orthogonally additive.
  Finally, since $G$ is compact and metrisable by \cite[Proposition 5.11]{robert83} we have that $\widehat{G}$ is countable. The measure $\mu$ is point supported so that Theorem~\ref{thm:ginvariantoa}~(ii) implies that $Q$ is strongly representable.
 \end{proof}

\end{example}

\section{Conclusions}\label{sec:conclu}
We conclude commenting on some aspects of uniqueness that naturally appear around the representation theorem proven.
Conditions \ref{H1}, \ref{H2} and \ref{H3} are sufficient conditions for an
Hermitian and closable quadratic form $(Q,\D)$ to be strongly representable by an operator $T$
(see Theorem~\ref{thm:mainthm1}).
It is clear that replacing \ref{H3} by the stronger condition requiring the semiboundedness of the forms $q_\alpha$ (cf. Corollary ~\ref{cor:mainthm}) and given $\H$ as a direct integral over a point supported measure space, the representing operator is uniquely determined. In fact, $T$ decomposes along the direct integral representation and the semibounded quadratic forms $q_\alpha$ are uniquely represented by the corresponding $T_\alpha$, $\alpha\in\mathcal{A}$. On the other hand, we have already seen in Subsection~\ref{subsec:position}
that the measure space $(\A,\Sigma(\A),\mu)$ is highly non unique and, in concrete examples, there are many ways to realise
the hypothesis \ref{H2} for a representing operator $T$. It remains an open question to decide under which conditions the operator $T$ strongly representing a Hermitian quadratic form $(Q,\D)$ is uniquely given.\\


\end{document}